\documentclass[a4paper,reqno, 12pt]{amsart}

\usepackage{mystyle}
\usepackage{fullpage}

\newcommand{\aff}{\mathrm{aff}}
\renewcommand{\sc}{\mathrm{sc}}
\newcommand{\OGr}{{\operatorname{OGr}}}

\newcommand{\Perm}{\mathrm{Perm}}
\newcommand{\Adm}{\mathrm{Adm}}
\newcommand{\sFl}{\mathscr{F}\!\ell}

\begin{document}

\title{Topological flatness of orthogonal spin local models}
\author{Jie Yang}
\address{Yau Mathematical Sciences Center, Tsinghua University, Haidian District, Beijing 100084, China}
\email{jie-yang@mail.tsinghua.edu.cn}

\begin{abstract}
     Let $p$ be an odd prime and $F$ be a complete discretely valued field with residue field of characteristic $p$. For any parahoric level structure of the split even orthogonal similitude group $\operatorname{GO}_{2n}$ over $F$, we prove a preliminary form of the Pappas--Rapoport flatness conjecture \cite[Conjecture 8.1]{pappas2009local}: the associated spin local model is topologically flat. 
\end{abstract}

\maketitle

\setcounter{tocdepth}{1}
\tableofcontents

\section{Introduction}
\subsection{}
Local models are projective schemes over the spectrum of a discrete valuation ring, that were first introduced to analyze the \etale-local structure of integral models of Shimura varieties. Rapoport and Zink \cite{rapoport1996period} constructed \dfn{naive local models} $\RM^\naive$ for Shimura varieties of PEL type via an explicit moduli interpretation. However, these naive local models are not always flat over the $p$-adic integers: Pappas \cite{pappas2000arithmetic}  first observed a failure of flatness in the case of ramified unitary groups, and Genestier later noted similar issues for split orthogonal groups. We refer to \cite{PRS13} for a comprehensive overview of the theory of local models.  

One remedy for the non-flatness of $\RM^\naive$ is to refine the (naive) moduli problem by imposing additional linear algebraic conditions, thereby cutting out a closed subscheme $\RM\sset\RM^\naive$ that is flat with the same generic fiber as $\RM^\naive$. 
Unlike the cases for symplectic groups and unitary groups (see \cite{gortz2001flatness,gortz2003flatness, pappas2000arithmetic, pappas2009local, smithling2015moduli, yu2019moduli, luo}, etc), local models for orthogonal groups have been less systematically investigated in the literature (but see \cite{smithling2011topological, he2020good, Zachos23}) for related results).

The subject of this paper is the case of PEL type D: local models for the split even orthogonal similitude group with \dfn{any} parahoric level structure.

Let $F$ be a complete discretely valued field with ring of integers $\CO$, uniformizer $\pi$, and residue field $k$ of characteristic $p> 2$. For an integer $n\geq 1$, set $V=F^{2n}$ and equip $V$ with a non-degenerate symmetric $F$-bilinear form $$\psi\colon V\times V\ra F.$$ Assume $\psi$ is split, i.e., $V$ admits an $F$-basis $(e_i)_{1\leq i\leq 2n}$ such that $\psi(e_i,e_j)=\delta_{i,2n+1-j}$. Denote by $$G\coloneqq \GO(V,\psi)$$ the (split) orthogonal similitude group associated with $(V,\psi)$. Let $G^\circ$ be the connected component of $G$ containing the identity element. 

Let $\CL$ be a self-dual periodic lattice chain of $V$ in the sense of \cite[Chapter 3]{rapoport1996period}. For each $\Lambda\in\CL$, we let $$\Lambda^\psi\coloneqq \cbra{x\in V\ |\ \psi(x,\Lambda)\sset \CO } $$ denote the dual lattice of $\Lambda$ with respect to $\psi$. Attached to $\CL$, one can define the naive local model $\RM^\naive_\CL$ following Rapoport--Zink \cite{rapoport1996period} (see Definition \ref{defn-naive} for the precise definition), via an explicit moduli functor. Then $\RM^\naive_\CL$ is a projective $\CO$-scheme, whose generic fiber is $\OGr(n,2n)_F$, the orthogonal Grassmannian of totally isotropic $n$-dimensional subspaces in $V$. In general, $\RM^\naive_\CL$ is not flat over $\CO$, see Genestier's example in \cite[\S 2.3]{PRS13}. In \cite[\S 7,8]{pappas2009local}, Pappas and Rapoport propose to add the \dfn{spin $\pm$-condition} to the moduli functor of $\RM^\naive_\CL$. The resulting moduli functor is representable by a closed subscheme $\RM^\pm_\CL$ of $\RM^\naive_\CL$ (see Definition \ref{defn-spin}). 

\begin{conj}[Pappas--Rapoport]  \label{conjPR} 
	The spin local model $\RM^\pm_\CL$ is flat over $\CO$. 
\end{conj}

\subsection{Main results}
For $n\leq 3$, Conjecture \ref{conjPR} can be verified by hand calculations. Throughout the paper, we assume $n\geq 4$. 

We say $P\sset G(F)$ is a \dfn{parahoric subgroup} of $G(F)$ if $P$ is a parahoric subgroup of $G^\circ(F)$ in the usual sense of Bruhat--Tits theory (\cite[\S 5.2]{bruhat1984groupes},\cite[\S 7.4]{kaletha2023bruhat}) for connected reductive groups. 
Our first result is the following description of $G^\circ(F)$-conjugacy classes of parahoric subgroups of $G(F)$.

\begin{thm}[{Proposition \ref{prop-Jform}}] \label{introthm-conjpara}
	Let $I$ be a non-empty subset of $[0,n]$ (notation as in \S \ref{subsec-notation-intro}). Define \begin{flalign*}
		\Lambda_i\coloneqq \CO\pair{\pi\inverse e_1,\ldots,\pi\inverse e_i,e_{i+1}, \ldots, e_{2n} }, \text{\ for $0\leq i\leq n$}. 
	\end{flalign*}
	Then the subgroup \begin{flalign*}
		P_I^\circ\coloneqq \cbra{g\in G(F)\ |\ g\Lambda_i=\Lambda_i, i\in I} \cap\ker\kappa,
	\end{flalign*}
	where $\kappa\colon G^\circ(F)\twoheadrightarrow \pi_1(G)$ denotes the Kottwitz map, is a parahoric subgroup of $G(F)$. 
	
	Furthermore, any parahoric subgroup of $G(F)$ is $G^\circ(F)$-conjugate to a subgroup $P_I^\circ$ for some (not necessarily unique) $I\sset [0,n]$. 
	The $G^\circ(F)$-conjugacy classes of maximal parahoric subgroups of $G(K)$ are in bijection with the set $\cbra{0,2,3,\ldots,\lfloor n/2\rfloor}$. The index set $\cbra{0}$ corresponds to a hyperspecial subgroup. 
\end{thm}

An analogous statement holds for the unitary similitude groups, see \cite[\S 1.2]{pappas2009local}.
The proof of Theorem \ref{introthm-conjpara} is based on the lattice-theoretic description of the Bruhat--Tits building for $G$; compare the unitary case in \cite{yang2024wildly}.

%\begin{remark}\label{intrormk13}
%	Denote $P_I\coloneqq \cbra{g\in G(F)\ |\ g\Lambda_i=\Lambda_i, i\in I}$. Then we have \begin{flalign*}
%		P_I/P_I^\circ \simeq \begin{cases}
%			\cbra{1} \quad &\text{if $0,n\in I$};\\ \BZ/2\BZ &\text{if $0\in I$ or $n\in I$};\\ \BZ/2\BZ\times \BZ/2\BZ &\text{if $0,n\notin I$}.
%		\end{cases}
%	\end{flalign*}  
%\end{remark}

The following theorem is a weaker version of Conjecture \ref{conjPR}. 
\begin{thm}[{Theorem \ref{coro-mainresults}}] \label{intro-thmtopo}
	The spin local model $\RM^\pm_\CL$ is topologically flat over $\CO$, i.e., the generic fiber $\RM_\CL^\pm\otimes F$ is dense in $\RM^\pm_\CL$ (with respect to the Zariski topology).
\end{thm}
\begin{remark}
   \begin{enumerate}
   	\item Theorem \ref{intro-thmtopo} is a generalization of \cite[Theorem 7.6.1]{smithling2011topological}, where Smithling proves the topological flatness in the Iwahori case. We emphasize that our approach differs from that of \cite{smithling2011topological}, as well as from the unitary analogues in \cite{smithling2011topological, smithling2011topological2}). Our method in fact give a different proof of \cite[Theorem 7.6.1]{smithling2011topological}, and provides more concrete geometric information (at least in the maximal parahoric case), which is used in \cite{yangflat}; see Theorem \ref{intro-thmpseumax}. The method is also flexible enough to be adapted to topological flatness questions in other settings; see \cite{yzz26}.
   	\item By Theorem \ref{intro-thmtopo}, Conjecture \ref{conjPR} reduces to verifying that the special fiber $\RM^\pm_{\CL,k}$ is reduced. See results in \cite{yangflat}.
   \end{enumerate}
\end{remark}

By Theorem \ref{introthm-conjpara}, any self-dual lattice chain $\CL$ is $G^\circ(F)$-conjugate to the standard self-dual lattice chain  $$\Lambda_I\coloneqq \cbra{\Lambda_\ell}_{\ell\in 2n\BZ\pm I}$$ for some $I\sset [0,n]$, where $\Lambda_{-i}\coloneqq \Lambda_i^\psi$ and $\Lambda_\ell\coloneqq \pi^{-d}\Lambda_{\pm i}$, for $i\in I$ and $\ell=2nd\pm i, d\in\BZ$. We have $\RM^\naive_\CL\simeq \RM^\naive_{\Lambda_I}$ and $\RM^\pm_\CL\simeq \RM^\pm_{\Lambda_I}$ (see Remark \ref{rmk-toI}). Set \begin{flalign*}
	\RM_I^\naive\coloneqq \RM^\naive_{\Lambda_I} \text{\ and\ } \RM^\pm_I\coloneqq \RM^\pm_{\Lambda_I}.
\end{flalign*}

If $\CL$ is $G^\circ(F)$-conjugate to $\Lambda_I$ for a singleton $I=\cbra{i}\sset [0,n]$, then we say $\CL$ is \dfn{pseudo-maximal}. In this case, we also write $\RM^\naive_i$ (resp. $\RM^\pm_i$) for $\RM_{\cbra{i}}^\naive$ (resp. $\RM^\pm_{\cbra{i}}$).

\begin{thm} \label{intro-thmpseumax}
	Suppose that $\CL$ is pseudo-maximal corresponding to $I=\cbra{i}$. 
	\begin{enumerate}
       \item The $\CO$-scheme $\RM^\naive_i$ (and hence $\RM^\pm_i$) is topologically flat. The underlying topological space of $\RM^\naive_i$ is the union of those of $\RM^+_i$ and $\RM^-_i$.
       \item If $i=0$ or $n$, then $\RM^\pm_i$ is isomorphic to a connected component of the orthogonal Grassmannian $\OGr(n,2n)$ over $\CO$. In particular, $\RM^\pm_i$ is irreducible and smooth of relative dimension $n(n-1)/2$.  
       
       If $i\neq 0,n$, then the special fiber $\RM^\pm_{i,k}$ has equidimension $n(n-1)/2$, and has exactly two irreducible components whose intersection is irreducible.
       \item There exists a stratification on the reduced special fiber \begin{flalign*}
           (\RM^\pm_{i,k})_\red =\coprod_{\ell=\max{\cbra{0,2i-n}}}^i\RM^\pm_i(\ell),
       \end{flalign*}
       where each stratum $\RM^\pm_i(\ell)$ (see Definition \ref{introdefn-Ml}) is a $k$-smooth locally closed subscheme of $\RM^\pm_{i,k}$, and $\RM^\pm_i(\ell')$ is contained in the closure $\ol{\RM^\pm(\ell)}$ if and only if $\ell'\leq \ell$.
   \end{enumerate} 
\end{thm}
Theorem \ref{intro-thmpseumax} follows from Proposition \ref{prop-i0n}, Corollary \ref{coro-topoflat} and Proposition \ref{prop-stratificationMpm}.
\begin{remark}
	We prove Theorem \ref{intro-thmtopo} using Theorem \ref{intro-thmpseumax}, and in fact only Part (1) of Theorem \ref{intro-thmpseumax} is needed. Parts (2) and (3) provide additional information on the geometry of $\RM^\pm_{i,k}$.
\end{remark}

\subsection{}
Our strategy for proving Theorem \ref{intro-thmtopo} is to first reduce to the pseudo-maximal case, namely Theorem \ref{intro-thmpseumax}. Let $I$ be a non-empty subset of $[0,n]$. Following \cite{gortz2001flatness,smithling2011topological}, we first embed the special fiber $\RM^\naive_{I,k}$ (resp. $\RM^\pm_{I,k}$) of $\RM^\naive_I$ (resp. $\RM^\pm_I$) in an affine flag variety $\sFl_I$, identifying $\RM^\naive_{I,k}$ (resp. $\RM^\pm_{I,k}$) with a union of Schubert cells in $\sFl_I$ (set-theoretically). Since the group $G$ is disconnected, 
%and the stabilizer subgroup $P_I$ may not be parahoric (see Remark \ref{intrormk13}), 
the situation is slightly more complicated than the unitary case. 

Let $\wt{W}$ (resp. $\wt{W}^\circ$) denote the Iwahori--Weyl group of $G$ (resp. $G^\circ$).  We prove (see Corollary \ref{coro-inWcirc} and \eqref{perm46}) that the Schubert cells in $\RM^\pm_{I,k}$ are in bijection with a double coset space $\Perm^\pm_I$ taken inside $W_I\backslash\wt{W}^\circ/W_I$. For each $i\in I$, let $$\rho_i\colon W_I\backslash \wt{W}^\circ/W_I\ra W_i\backslash\wt{W}^\circ/W_i$$ denote the obvious projection map.
%and that there exists a commutative diagram (by Lemma \ref{lem-intersection} and \eqref{eq48}) \begin{flalign*}
%	\begin{split}
%		\xymatrix{
%		   \Perm_I^\pm\ar[r]\ar[d]^{\rho_i} &\RM_{I,k}^\pm\ar[d]^{q_i}\\ \Perm_i^\pm\ar[r] &\RM_{i,k}^\pm \ ,
%		}
%	\end{split}  
%\end{flalign*} 
%where $i\in I$, and $\rho_i, q_i$ denote the obvious projection maps, and the horizontal maps send the source to the corresponding Schubert cells in $\RM^\pm_{I,k}$ or $\RM^\pm_{i,k}$. 
We obtain that \begin{flalign*}
	\Perm_I^\pm =\bigcap_{i\in I}\rho_i\inverse(\Perm_i^\pm).
\end{flalign*} 

By Theorem \ref{intro-thmpseumax} (1), $\Perm_i^\pm$ equals the admissible set $\Adm_i(\mu_\pm)\coloneqq W_i\backslash\Adm(\mu_\pm)/W_i$, where \begin{flalign}
    \mu_+\coloneqq (1^{(n)},0^{(n)}) \text{\ and\ } \mu_-\coloneqq (1^{(n-1)},0,1,0^{(n-1)}) \label{1cochar}
\end{flalign}
are two minuscule cocharacters of $G$.
Then Theorem \ref{intro-thmtopo} follows from the vertexwise criteria \cite[Theorem 1.5]{haines2017vertexwise} for the admissible sets. For further details of the reduction step, see \S \ref{subsec-genepara}.
 
We now explain the proof of Theorem \ref{intro-thmpseumax} in greater detail. For $I=\cbra{i}$, we compute the double coset space $\Perm_i^\pm$ parametrizing the Schubert cells in $\RM^\pm_{i,k}$. In this situation, $\RM^\naive_i$ is a union of $\RM^+_i$ and $\RM^-_i$, and $\RM^\naive_{i,k}$ consists of exactly $\min{\cbra{i,n-i}}+4$ Schubert cells. We explicitly write down representatives for each of these cells, and verify that all representatives lift to the generic fiber $\RM^\naive_{i,F}$. This proves the density of $\RM^\naive_{i,F}$ in $\RM^\naive_i$, and hence Theorem \ref{intro-thmpseumax} (1).
 
\begin{defn}[{Definition \ref{defn-Mell}}]  \label{introdefn-Ml}
    Let $\iota\colon \Lambda_i\ra \pi\inverse\Lambda_i^\psi=\Lambda_{2n-i}$ denote the natural inclusion map (and its base change). 
    For an integer $\ell$, denote by $$\RM^\pm_i(\ell)\sset \RM^\pm_{i,k}$$ the locus where $\iota(\CF_i)$ has rank $\ell$.
\end{defn}
By standard properties of the rank function, these loci provide a stratification of $(\RM^\pm_{i,k})_\red$ as in Theorem \ref{intro-thmpseumax} (3), see Proposition \ref{prop-stratificationMpm}. Note that $\RM^\pm_{i,k}$ is an empty set if $\ell<\max\cbra{0,2i-n}$ or $\ell>i$. We show that if $i\neq 0,n$, then  $\RM^\pm_i(i)$ consists of two Schubert cells, while $\RM^\pm_i(\ell)$ for $\max{\cbra{0,2i-n}}\leq \ell<i$ is a single Schubert cell. This implies parts (2) and (3) of Theorem \ref{intro-thmpseumax}.

\subsection{Organization}
We now give an overview of the paper. 

In \S \ref{sec-para}, we explicitly describe the Bruhat--Tits building for $G$ in terms of norms and lattice chains. As a corollary, we describe $G^\circ(F)$-conjugacy classes of parahoric subgroups as in Theorem \ref{introthm-conjpara}. 

In \S \ref{sec-naivespinmod}, we construct naive and spin local models using explicit moduli functors, following Pappas--Rapoport \cite[\S 8]{pappas2009local}. We embed the special fiber $\RM^\naive_{I,k}$ into an affine flag variety $\sFl _I$, identifying $\RM^\naive_{I,k}$ and $\RM^\pm_{I,k}$ with unions of Schubert cells in $\sFl_I$. 

In \S \ref{sec-topoflat}, we prove results for the topological flatness, namely  Theorem  \ref{intro-thmtopo} and \ref{intro-thmpseumax}. In \S \ref{subsec-pseud}, we consider the pseudo-maximal case, that is, $I=\cbra{i}$ is a singleton. We show that in this case $\RM^\naive_{i}$ is already topologically flat. More precisely, we obtain that there are exactly $\min{\cbra{i,n-i}}+4$ Schubert cells in $\RM^\naive_{i,k}$. By explicitly writing down representatives of each Schubert cell, we show that these representatives all lift to the generic fiber $\RM^\naive_{i,F}$, thereby proving Theorem \ref{intro-thmpseumax} (1). In addition, we construct a stratification of the special fiber of the spin local model $\RM^\pm_i$. By relating each stratum to Schubert cells, we prove Theorem \ref{intro-thmpseumax} (2) and (3).  In \S \ref{subsec-genepara}, we prove that $\RM^\pm_I$ is topologically flat for any $I$, which implies Theorem \ref{intro-thmtopo}. The argument proceeds by reduction to the pseudo-maximal case. The key ingredient here is the vertexwise criteria \cite{haines2017vertexwise} for the admissible sets.

\subsection*{Acknowledgements}
I thank I. Zachos and Z. Zhao for helpful discussions while preparing this paper. I am also grateful to Y. Luo for comments on a preliminary draft. I thank the referee for valuable comments and suggestions. This work is partially supported by the Shuimu Tsinghua Scholar Program of Tsinghua University.

\subsection{Notation} \label{subsec-notation-intro}
Throughout the paper, $(F,\CO,\pi, k)$ denotes a complete discretely valued field with ring of integers $\CO$, uniformizer $\pi$, and residue field $k$ of characteristic $p>2$. Denote by $\ol{k}$ an algebraic closure of $k$. We also employ an auxiliary complete discretely valued field $K$ with ring of integers $\CO_K$, uniformizer $t$, and the same residue field $k$; eventually $K$ will be the field $k((t))$ of Laurent series over $k$.

For $1\leq i\leq 2n$, define $i^*\coloneqq 2n+1-i$. Given $v\in\BZ^{2n}$, we write $v(i)$ for the $i$-th entry of $v$, and $\Sigma v$ for the sum of entries of $v$. Let $v^*\in\BZ^{2n}$ denote the vector defined by $v^*(i)\coloneqq v(i^*)$ for $1\leq i\leq 2n$. For $v,w\in\BZ^{2n}$, we write $v\geq w$ if $v(i)\geq w(i)$ for all $1\leq i\leq 2n$. For $d\in \BZ$, denote by $\mathbf d$ the image of $d$ along the diagonal embedding $\BZ\ra \BZ^{2n}$. The expression $(a^{(r)},b^{(s)},\ldots)$ denotes the vector with $a$ repeated $r$ times, followed by $b$ repeated $s$ times, and so on. 

For any real number $x$, denote by $\lfloor x\rfloor$ (resp. $\lceil x\rceil$) for the greatest (resp. smallest) integer $\leq x$ (resp. $\geq x$).

For integers $n_2\geq n_1$, denote $[n_1,n_2]\coloneqq \cbra{n_1,\ldots,n_2}$.  For a subset $E\sset [1,2n]$ of cardinality $n$, set $E^*\coloneqq \cbra{i^*\ |\ i\in E}$ and $E^\perp\coloneqq (E^*)^c$ (the complement of $E^*$ in $[1,2n]$). Denote by $\Sigma E$ the sum of the entries in $E$.  

For a scheme $\CX$ over $\CO$, we let $\CX_k$ or $\CX\otimes k$ denote the special fiber $\CX\otimes_\CO k$, and $\CX_F$ or $\CX\otimes F$ denote the generic fiber $\CX\otimes_\CO F$.

\section{Parahoric subgroups of the orthogonal similitude groups} \label{sec-para}
We follow the notation as in \S \ref{subsec-notation-intro}. In particular, $K$ is a complete discretely valued field. Denote by $\val\colon K\cross\ra \BZ$ the normalized valuation on $K$. 

\subsection{Orthogonal similitude group}
Let $n\geq 4$ be an integer.
Let $V=K^{2n}$ be a $K$-vector space of dimension $2n$ equipped with a symmetric $K$-bilinear form $\psi$. Assume that there exists a $K$-basis $(e_i)_{1\leq i\leq 2n}$ of $V$ such that $\psi(e_i,e_j)=\delta_{i,2n+1-j}$. Denote by $G=\GO(V,\psi)$ the associated split orthogonal similitude group. Let $G^{\circ}$ be the connected component of $G$ containing the identity element. Following \cite[\S 5]{smithling2011topological}, let $\tau$ denote the permutation matrix in $G(K)$ corresponding to the transposition $(n,n+1)$. Then we have \begin{flalign}
	G = G^{\circ}\sqcup \tau G^{\circ}.  \label{Gtwocompo}
\end{flalign}   

Let $T$ be the standard split maximal torus of diagonal matrices in $G$. For a $K$-algebra $R$, we have \begin{flalign*}
	 T(R)=\cbra{\diag(x_1,\ldots,x_{2n})\in\GL_{2n}(R)\ |\ x_1x_{2n}=x_2x_{2n-1}=\cdots=x_nx_{n+1} }.
\end{flalign*}
Then the group $X_*(T)$ of cocharacters of $T$ is isomorphic to \begin{flalign*}
	\cbra{(r_1,\ldots,r_{2n})\in\BZ^{2n}\ |\ r_1+r_{2n}=\cdots=r_n+r_{n+1} }.
\end{flalign*}
We have an isomorphism \begin{flalign}
    \begin{split}
    	T(K)/T(\CO_K) &\simto X_*(T)\\ \diag(x_1,\ldots,x_{2n}) &\mapsto (\val(x_1),\ldots,\val(x_{2n})).  \label{Tlattice}
    \end{split}
\end{flalign}
The coroot lattice $Q^\vee\sset X_*(T)$ can be identified with 
%\[\cbra{(r_1,\ldots,r_{2n})\in\BZ^{2n}\ \vline \ \Centerstack[l]{$r_1+r_{2n}=\cdots=r_n+r_{n+1}=0$ \\ $r_1+\cdots+r_n$ is even } }. \]

\[
\left\{ (r_1,\ldots,r_{2n})\in\mathbb Z^{2n}\ \middle|\ 
\begin{array}{l}
r_1+r_{2n}=\cdots=r_n+r_{n+1}=0,\\
r_1+\cdots+r_n \text{ is even}
\end{array}
\right\}.
\]

Let $\sB=\sB(G,K)$ denote the (extended) Bruhat--Tits building of $G$. The standard apartment $\CA$ associated with $T$ can be identified with \begin{flalign*}
	  X_*(T)\otimes_\BZ\BR \simeq \cbra{(r_1,\ldots,r_{2n})\in\BR^{2n}\ |\ r_1+r_{2n}=\cdots=r_n+r_{n+1} }.
\end{flalign*}
We fix the base alcove $\fa$ as in \cite[\S 8.1]{smithling2014topological}. Then $\fa$ has  $n+1$ vertices (minimal facets) $a+\BR\cdot(1,\ldots,1)$, for $a$ one of the following points  \begin{equation} \label{vertices}
	\begin{split}
		a_0 &\coloneqq (0,\ldots,0),\\ a_{0'}&\coloneqq (-1,0^{(2n-2)},1),\\ a_i&\coloneqq \rbra{(-1/2)^{(i)}, 0^{(2n-2i)},(1/2)^{(i)} },\ 2\leq i\leq n-2,\\ a_n&\coloneqq \rbra{(-1/2)^{(n)},(1/2)^{(n)}},\\ a_{n'}&\coloneqq \rbra{(-1/2)^{(n-1)},1/2,-1/2,(1/2)^{(n-1)} }.
	\end{split}
\end{equation}
The vertices corresponding to $a_0,a_{0'},a_n,a_{n'}$ are hyperspecial; the other vertices are nonspecial.

Define the Iwahori-Weyl groups $$\wt{W}\coloneqq N_{G}(K)/T(\CO_K),\quad \wt{W}^\circ\coloneqq N_{G^{\circ}}(K)/T(\CO_K), $$ where $N_{G}$ (resp. $N_{G^{\circ}}$) is the normalizer of $T$ in $G$ (resp. $G^{\circ}$). Then \[\wt{W}=\wt{W}^\circ\sqcup \tau\wt{W}^\circ. \]
As usual, $\wt{W}$ admits two semi-direct product decompositions (see \cite[\S 5.3]{smithling2011topological}) \begin{flalign*}
    X_*(T)\rtimes W \text{\ and\ } W_{\aff}\rtimes \Omega,
\end{flalign*} 
where $W=N_G(K)/T(K)$ denotes the (finite) Weyl group, $W_{\aff}$ denotes the affine Weyl group, and $\Omega$ denotes the stabilizer subgroup of $\fa$. We also have $\wt{W}^\circ\simeq X_*(T)\rtimes W^\circ$, where $W^\circ=N_{G^\circ}(K)/T(K)$. Concretely, we may identify \begin{flalign*}
	W^\circ\simeq S_{2n}^\circ \text{\ and\ } W\simeq S_{2n}^*,
\end{flalign*} 
where $S_{2n}^*$ (resp. $S_{2n}^\circ$) denotes the subgroup of $S_{2n}$ consisting of (resp. even) permutations $\sigma$ satisfying \[\sigma(2n+1-i)+\sigma(i)=2n+1 \text{\ for all $1\leq i\leq 2n$}. \]
We can also identify 
\begin{flalign}
    \wt{W} &\simeq \cbra{(r_1,\ldots,r_{2n})\in\BZ^{2n}\ |\ r_1+r_{2n}=r_2+r_{2n-1}=\cdots=r_n+r_{n+1} }\rtimes S_{2n}^*, \label{weyl} \\ 
    \wt{W}^\circ &\simeq \cbra{(r_1,\ldots,r_{2n})\in\BZ^{2n}\ |\ r_1+r_{2n}=\cdots=r_n+r_{n+1} }\rtimes S_{2n}^\circ,  \label{Wcirc}\\
     W_\aff &\simeq Q^\vee\rtimes W^0\simeq \cbra{(r_1,\ldots,r_{2n})\in\BZ^{2n}\ \vline \ \begin{array}{l}
	r_1+r_{2n}=\cdots=r_n+r_{n+1}=0\\ \text{$r_1+\cdots+r_n$ is even}
\end{array} }\rtimes S_{2n}^\circ.
\end{flalign}

Note that the group $\wt{W}$ acts on $\BR^{2n}$ via affine transformations: for any $v\in \BR^{2n}$ and $w=t^ww_0$, where $w_0\in S_{2n}^*$ and $t^w\in\BZ^{2n}$ as in \eqref{weyl}, define \begin{flalign}
    wv\coloneqq w_0v+t^w,   \label{actionW}
\end{flalign}
where $w_0$ acts on $v$ by permuting the coordinates: $(w_0v)(i)\coloneqq v(w_0\inverse i)$ for $1\leq i\leq 2n$.

Let \begin{flalign}
	\kappa\colon G^{\circ}(K)\twoheadrightarrow \pi_1(G^{\circ})\simeq \BZ\oplus \BZ/2\BZ  \label{eq-kott}
\end{flalign} be the Kottwitz homomorphism for $G^{\circ}(K)$, cf. \cite[\S 4.3]{smithling2011topological}.  Then $\kappa$ can be characterized as follows: for $g\in G^{\circ}(K)$, with respect to the Cartan decomposition $$G^{\circ}(K)=G^{\circ}(\CO_K)T(K)G^{\circ}(\CO_K),$$ write $g=g_1tg_2$ with $g_1,g_2\in G^{\circ}(\CO_K)$ and $t=\diag(t_1,\ldots,t_{2n})\in T(K)$, then $$\kappa(g)\coloneqq \rbra{\val(c(g)),\sum_{j=1}^{n}\val(t_j)\mod 2},$$ where $c\colon G^\circ\ra \BG_m$ denotes the similitude character.

%Note that $G(K)$ and $G^{\circ}(K)$ act on the building $\sB$.  Let $P$ (resp. $Q$) be the stabilizer subgroup in $G(K)$ (resp. $G^{\circ}(K)$) of the vertex \begin{flalign}
%	a_m=\rbra{(-1/2)^{(m)},0^{(n)},(1/2)^{m)} }. \label{amequa}
%\end{flalign}  Set $P^\circ\coloneqq Q\cap \ker\kappa$. Then $P/Q\simeq\BZ/2\BZ$ is generated by $\tau=(n,n+1)$, and $Q/P^0\simeq\BZ/2\BZ$ is generated by $g_0\coloneqq \diag(t\inverse,1^{(2n-2)},t)\cdot (1,2n+1)(n,n+1)$, where $(1,2n+1)(n,n+1)$ denotes the permutation matrix in $G^{\circ}(K)$. 

The Kottwitz homomorphism $\kappa$ induces a map \begin{flalign*}
	\kappa'\colon \wt{W}^\circ\ra \BZ\oplus\BZ/2\BZ.
\end{flalign*}
With respect to the identification \eqref{Wcirc}, $\kappa'$ sends the element $(r_1,\ldots,r_{2n})\sigma\in\wt{W}^\circ$ to $(r_1+r_{2n},\sum_{j=1}^nr_j\mod 2)\in \BZ\times\BZ/2\BZ$.
\begin{defn} \label{defnWprime}
    Denote $\varepsilon\colon \wt{W}^\circ\xrightarrow{\kappa'}\BZ\oplus\BZ/2\BZ\twoheadrightarrow\BZ/2\BZ$.
	Set \begin{flalign}
		{W}'\coloneqq \ker\varepsilon\sset \wt{W}^\circ. \label{Wprime}
	\end{flalign}
\end{defn}
By construction, we have \begin{flalign*}
	W' \simeq \cbra{(r_1,\ldots,r_{2n})\in\BZ^{2n}\ \vline \ \Centerstack[l]{$r_1+r_{2n}=\cdots=r_n+r_{n+1}$ \\ $r_1+\cdots+r_n$ is even } }\rtimes S_{2n}^\circ.
\end{flalign*}

\subsection{Bruhat--Tits theory for $G$}
In this subsection, we describe the Bruhat--Tits building $\sB$ in terms of norms and self-dual lattice chains. This is a variant of a special case of the results in \cite{bruhat1987schemas} (or \cite[\S 15]{kaletha2023bruhat}) that describe $\sB(O(n),K)$.

Recall that a \dfn{norm} on $V$ is a map $\alpha\colon V\ra \BR\cup\cbra{+\infty}$ such that for $x,y\in V$ and $\lambda\in K$, we have \begin{flalign*}
	   \alpha(x+y) \geq \inf\cbra{\alpha(x),\alpha(y)},\ \alpha(\lambda x)=\val(\lambda)+\alpha(x), \text{\ and\ } x=0 \Leftrightarrow \alpha(x)=+\infty.
\end{flalign*}
The set $\CN$ of norms on $V$ carries a natural action of $\GL(V)(K)$: for $g\in \GL(V)(K)$ and $\alpha\in\CN$, define $g\alpha(x)\coloneqq \alpha(g\inverse x)$, for $x\in V$. 
It is well-known (see \cite[Proposition 15.1.31, Remark 15.1.32]{kaletha2023bruhat}) that there exists a $\GL(V)(K)$-equivariant bijection $$\sB(\GL(V),K)\simto \CN$$ between the Bruhat--Tits building $\sB(\GL(V),K)$ and $\CN$. For a norm $\alpha$ on $V$, we define the \dfn{dual norm} $\alpha^\vee$ via setting $$\alpha^\vee(x)\coloneqq \inf_{y\in V}\cbra{\val(\psi(x,y))-\alpha(y)}, \text{\ for $x\in V$}.$$
We say $\alpha$ is \dfn{almost self-dual} if $\alpha^\vee=\alpha+r$ for some constant $r\in \BR$. If moreover $r=0$, we say $\alpha$ is \dfn{self-dual}.

\begin{prop}\label{prop-Bnorms}
	There exists a $G(K)$-equivariant injection \begin{flalign*}
		\sB\hookrightarrow \sB(\GL(V),K)\simeq\CN
	\end{flalign*} whose image is the set of almost self-dual norms on $V$.  Furthermore, for $$a=(r_1,\ldots,r_{2n})  \in \CA,$$ the corresponding element in ${\CN}$ is the almost self-dual norm \begin{flalign}
		\alpha_a\colon V\ra \BR\cup\cbra{+\infty},\quad   \sum_{i=1}^{2n}x_ie_i \mapsto \inf\tcbra{\val(x_i)-r_i\ |\ 1\leq i\leq 2n}.   \label{alphaa}
	\end{flalign}
	We have $\alpha_a^\vee=\alpha_a+(r_1+r_{2n})$.
\end{prop}
\begin{proof}
	Denote by $H\coloneqq \mathrm{O}(V,\psi)$ the orthogonal group associated to $(V,\psi)$. Since $p\neq 2$, by \cite[Proposition 15.2.14 (3)]{kaletha2023bruhat}, the image of the inclusion $$\iota\colon \sB(H,K)\hookrightarrow \sB(\GL(V),K)\simeq \sN$$ is the set of self-dual norms on $V$. Note that $\sB(G,K)=\sB(H,K)\times X_*(Z_G)_\BR\simeq \sB(H,K)\times\BR$ (cf. \cite[\S 2.1]{tits1979reductive}). Here, $g\in G(K)$ acts on $(x,r)\in \sB(H,K)\times \BR$ via \begin{flalign*}
		g.(x,r) = (g.x, r-\val(c(g))/2),
	\end{flalign*} where the first factor denotes the usual $G(K)$-action on $\sB(H,K)\simeq\sB(G_\ad,K)$ through $G(K)\ra G_\ad(K)$; cf. \cite[\S 4.3]{kaletha2023bruhat}. We define \begin{equation} \label{iota1}
		\begin{split}
			\iota'\colon \sB=\sB(H,K)\times \BR &\lra \CN\\ (x, r) &\mapsto \iota(x)-r.
		\end{split}
	\end{equation}
	Set $T_H\coloneqq T\cap H$. Then $T_H$ is a maximal split torus in $H$, and \begin{flalign*}
		X_*(T_H)\simeq \cbra{(s_1,\ldots,s_{2n})\in\BZ^{2n}\ |\ s_1+s_{2n}=\cdots=s_n+s_{n+1}=0}.
	\end{flalign*}
	Let $b=(s_1,\ldots,s_{2n})$ be a point in the apartment  $X_*(T_H)\otimes\BR\sset \sB(H,K)$. By \cite[Lemma 15.2.11 (2),(3)]{kaletha2023bruhat}, the norm $\iota(b)$ is given by \begin{flalign}
		   \sum_{i=1}^{2n}x_ie_i\mapsto \inf\cbra{\val(x_i)-s_i\ |\ 1\leq i\leq 2n }.  \label{norm2}
	\end{flalign}
	Let $a=(r_1,\ldots,r_{2n})\in\CA$. Set $r\coloneqq (r_1+r_{2n})/2$. Write $a=(a-r,r)\in \sB(H,K)\times \BR$. Using \eqref{iota1} and \eqref{norm2}, we obtain   the description of the associated norm $\alpha_a$ as in \eqref{alphaa}. 
	By construction, we have $\alpha_{a-r}=\alpha_a+r$. Since $\alpha_{a-r}\in \sB(H,K)$ is self-dual, we obtain that \begin{flalign*}
		\alpha_a^\vee = (\alpha_{a-r}-r)^\vee=\alpha_{a-r}^\vee+r=\alpha_{a-r}+r=\alpha_a+2r=\alpha_a+(r_1+r_{2n}).
	\end{flalign*}
\end{proof}

Recall that a \dfn{(periodic) lattice chain} of $V$ is a non-empty set $L_\bullet$ of lattices in $V$ such that lattices in $L_\bullet$ are totally ordered with respect to the inclusion relation, and $\lambda L\in L_\bullet$ for $\lambda\in K\cross$ and $L\in L_\bullet$.  A \dfn{graded lattice chain} is a pair $(L_\bullet,c)$, where $L_\bullet$ is a lattice chain of $V$ and $c\colon L_\bullet\ra \BR$ is a strictly decreasing function such that for any $\lambda\in K$ and $L\in L_\bullet$, we have $$c(\lambda L)=\val(\lambda)+c(L).$$ The function $c$ is called a \dfn{grading} of $L_\bullet$. Denote by $\CGLC$ the set of graded lattice chains of $V$. Then $\CGLC$ carries a natural $\GL(V)(K)$-action: for $(L_\bullet,c)\in \CGLC$ and $g\in\GL(V)(K)$, define $g(L_\bullet,c)\coloneqq (gL_\bullet,gc)$, where $gL_\bullet$ consists of lattices of the form $gL$  for $L\in L_\bullet$, and $(gc)(gL)\coloneqq c(L)$ for $L\in L_\bullet$. 

Let $(L_\bullet,c)$ be a graded lattice chain. The dual graded lattice chain is $(L_\bullet^\vee,c^\vee),$ where $L_\bullet^\vee$ is the set of the lattices of the form $L^\vee\coloneqq \cbra{x\in V\ |\ \psi(x,L)\sset \CO_K}$ for $L\in L_\bullet$, and $c^\vee(L^\vee)\coloneqq -c(L^-)-1,$ where $L^-$ is the smallest member of $L_\bullet$ that properly contains $L$. We say $(L_\bullet,c)$ is \dfn{almost self-dual} if $(L_\bullet,c)=(L_\bullet^\vee,c^\vee)+r$ for some constant $r\in\BR$. If moreover $r=0$, then we say $(L_\bullet,c)$ is self-dual.

\begin{prop}\label{prop-grlattice}
   \begin{enumerate}
   	\item There exists a $\GL(V)(K)$-equivariant bijection between $\CN$ and $\CGLC$. More precisely, given $\alpha\in\CN$, we can associate a graded lattice chain $(L_\alpha,c_\alpha)$, where $L_\alpha$ is the set of the  following lattices $$L_{\alpha,r}=\cbra{x\in V\ |\ \alpha(x)\geq r}, \text{\ for\ }r\in \BR,$$ and the grading $c_\alpha$ is defined by $$c_\alpha(L_{\alpha,r})=\inf_{x\in L_{\alpha,r}}\alpha(x).$$
	Conversely, given a graded lattice chain $(L_\bullet,c)\in\CGLC$, we can associate a norm $$\alpha_{(L_\bullet,c)}(x)\coloneqq  \sup\cbra{c(L)\ |\ x\in L \text{\ and\ }L\in L_\bullet}.$$
	We say that the norm $\alpha$ and the graded lattice chain $(L_\alpha,c_\alpha)$ in the above bijection \dfn{correspond to} each other.
	\item There exists a $G(K)$-equivariant bijection between $\sB$ and the set of almost self-dual graded lattice chains. 
   \end{enumerate}
\end{prop}
\begin{proof}
	(1) See \cite[Proposition 15.1.21]{kaletha2023bruhat}. 
	
	(2) It follows from (1) and Proposition \ref{prop-Bnorms}.
\end{proof}

\subsection{Parahoric subgroups}
\begin{defn}
     Set \begin{flalign*}
     	\lambda_{0'} &\coloneqq \CO_K\pair{t\inverse e_1,e_2,\ldots,e_{2n-1},t e_{2n}}, \\
     	\lambda_i &\coloneqq \CO_K\pair{t\inverse e_1,\ldots,t\inverse e_i,e_{i+1},\ldots,e_{2n}}, 0\leq i\leq n \text{\ and\ } i\neq 1,n-1, \\
     	\lambda_{n'} &\coloneqq \CO_K\pair{t\inverse e_1,\ldots,t\inverse e_{n-1}, e_n, t\inverse e_{n+1},e_{n+2},\ldots,e_{2n}}.
     \end{flalign*}
     Denote \begin{flalign*}
     	\sI\coloneqq \cbra{0,0',2,3,\ldots,n-2,n,n'}.
     \end{flalign*} 
     For $j\in \sI$, denote by $\lambda_{-j}=\lambda_j^\psi$ the dual lattice of $\lambda_j$ with respect to $\psi$.  We have \begin{flalign*}
     	   \lambda_{-j}\sset \lambda_j\sset t\inverse\lambda_{-j}.
     \end{flalign*}
     Note that $\lambda_0=\lambda_{-0}$, $\lambda_{0'}=\lambda_{-0'}$, $\lambda_{-n}=t\lambda_n$, and $\lambda_{-n'}=t\lambda_{n'}$.
     For $\ell=2nd\pm j$ for some $d\in\BZ$ and $j\in \sI$, define $\lambda_\ell\coloneqq t^{-d}\lambda_{\pm j}$. We obtain a self-dual lattice chain \begin{flalign*}
     	 \lambda_{\cbra{j}}\coloneqq \cbra{\lambda_\ell}_{\ell\in \cbra{2n\BZ\pm j}}.
     \end{flalign*}
\end{defn}

The following lemma is straightforward to verify. 
\begin{lemma} \label{lem-normlattice}
    \begin{enumerate}
    	\item Under the identification in Proposition \ref{prop-Bnorms}, the point $a_j$ ($j\in\sI$) in \eqref{vertices} corresponds to the self-dual norm  \begin{flalign*}
    		  \alpha_j\colon V \lra \BR\cup\cbra{+\infty},\quad \sum_{i=1}^{2n}x_ie_i \mapsto \inf\cbra{\val(x_i)-a_j(i)\ |\ 1\leq j\leq 2n }.
    	\end{flalign*} 
    	\item Under the identifications in Proposition \ref{prop-Bnorms} and Proposition \ref{prop-grlattice} (1), the self-dual graded lattice chain corresponding to the vertex $a_j$ ($j\in\sI'$) in \eqref{vertices} is given by \begin{flalign*}
		  (\lambda_{\cbra{j}}, c_{\cbra{j}}),
	\end{flalign*}
	where the grading $c_{\cbra{j}}$ is 
	\begin{flalign*}
		  c_{\cbra{j}}(\lambda_\ell) = \begin{cases}
		  	 -d -1/2 \quad &\text{if $\ell=2nd+j$ and $j\neq 0,0'$,}\\ -d &\text{if $\ell=2nd-j$ and $j\neq n,n'$}.
		  \end{cases}
	\end{flalign*}
    \end{enumerate}
\end{lemma}

\begin{corollary} \label{coro-stabvi}
    The stabilizer subgroup of $a_j$ ($j\in\sI$) in $G(K)$ equals \begin{flalign*}
    	P_{\cbra{j}}\coloneqq \cbra{g\in G(K)\ |\ g\lambda_j=\lambda_j}.
    \end{flalign*}	
\end{corollary}
\begin{proof}
	Since the bijection between $\sB$ and the set of almost self-dual graded lattice chains is $G(K)$-equivariant by Proposition \ref{prop-grlattice}, we obtain that the stabilizer subgroup of $a_j$ equals the stabilizer subgroup $C$ of $(\lambda_{\cbra{j}},c_{\cbra{j}})$. For $g\in C$, the condition that $g$ preserves the grading $c_{\cbra{j}}$ amounts to that $g$ sends $\lambda\in \lambda_{\cbra{j}}$ to $\lambda$. Therefore, the stabilizer subgroup equals \begin{flalign*}
		 \cbra{g\in G(K)\ |\ g\lambda_j=\lambda_j\text{\ and\ } g\lambda_{-j}=\lambda_{-j}} = P_{\cbra{j}}.
	\end{flalign*} 
\end{proof}

Recall that $\kappa\colon G^\circ(K)\ra \pi_1(G^\circ)$ denotes the Kottwitz homomorphism for $G^\circ$.  
\begin{defn}
%	Denote $\sI\coloneqq \cbra{0,1,\ldots,n}$. Set \begin{flalign*}
%		\Sigma\coloneqq \cbra{I\sset \sI\ |\ \Centerstack[l]{$I$ is non-empty, and if $1\in I$, then $0\in I$,  \\ if $n-1\in I$, then $n\in I$. }  }.
%	\end{flalign*}
	For a non-empty subset $J\sset \sI$, set \begin{flalign}
		P_J\coloneqq \cbra{g\in G(K)\ |\ g\lambda_j=\lambda_j, j\in J} \text{\ and\ } P_J^\circ\coloneqq P_J\cap \ker \kappa. \label{PJsubgroup}
	\end{flalign}
	If $J=\cbra{j}$ is a singleton, we simply write $P_j$ for $P_{\cbra{j}}$.
\end{defn}

\begin{defn}
	We say $P\sset G(K)$ is a \dfn{parahoric subgroup of $G(K)$} if $P$ is a parahoric subgroup of $G^\circ(K)$ in the usual sense of Bruhat--Tits theory (for connected reductive groups).  
\end{defn}

\begin{prop} \label{prop-paraIndex}
    \begin{enumerate}
    	\item The subgroup $P_J^\circ$ in \eqref{PJsubgroup} is a parahoric subgroup of $G(K)$. Any parahoric subgroup of $G(K)$ is $G^\circ(K)$-conjugate to a subgroup $P_J^\circ$ for some (not necessarily unique) $J\sset \sI$. 
    	\item The $G^\circ(K)$-conjugacy classes of maximal parahoric subgroups of $G^\circ(K)$ are in bijection with the set $\tcbra{P^\circ_{j}\ |\ j=0 \text{\ or\ }2\leq j\leq \lfloor n/2\rfloor }$.
    \end{enumerate}
\end{prop}
\begin{proof}
   (1)
%	Set \begin{flalign*}
%		a_1\coloneqq \rbra{-1/2,0^{(2n-2)},1/2} \text{\ and\ } a_{n-1}\coloneqq \rbra{(-1/2)^{(n-1)},0,0,(1/2)^{(n-1)}}.
%	\end{flalign*}
%	Then $a_1$ is the midpoint of $a_0$ and $a_{0'}$; $a_{n-1}$ is the midpoint of $a_n$ and $a_{n'}$. 
    Note that any non-empty $J\sset \sI$ determines a facet $\ff_J$ in the base alcove $\fa$, where the vertices of the closure $\ol{\ff}_J$ are given by $\cbra{a_j+\BR(1,\ldots,1)\ |\ j\in J }$, notation as in \eqref{vertices}. 
	 By Corollary \ref{coro-stabvi}, the (pointwise) stabilizer subgroup of $\ff_J$ is equal to $P_J$. 
	Therefore,  $P_J^\circ$ is a parahoric subgroup of $G(K)$ by \cite[Proposition 3]{haines2008parahoric}.
	
	Note that any facet in the building $\sB$ is $G^\circ(K)$-conjugate to a facet $\ff$ in the alcove $\fa$. Set $J\coloneqq \tcbra{j\in \sI\ |\ a_j\in \ov{\ff}}$. Then the parahoric subgroup attached to $\ff$ is equal to $P^\circ_J$. 
	
	(2) Let $G_{\mathrm{sc}}$ denote the simply connected cover of the derived subgroup $G_\der$. Let $T_{\mathrm{sc}}$ be the maximal torus of $G_\sc$.    Let $H_0$ denote the image of $T_{\mathrm{sc}}(K)$ in $T(K)$. Then \begin{flalign*}
		 H_0=(\ker\kappa)\cap T_\der(K).
	\end{flalign*} 
	Let $H_1$ (resp. $H_2$) be the subgroup of $T(K)$ generated by \begin{flalign*}
		t_1\coloneqq \diag(t\inverse,1^{(2n-2)},t)\ (\text{resp.\ } t_2\coloneqq \diag(1^{(n)},t^{(n)})). 
	\end{flalign*} 
	Note that the isomorphism in \eqref{Tlattice} takes $H_0T(\CO_K)$ to the coroot lattice $Q^\vee\sset X_*(T)$.
	Since the image of $H_0H_1H_2\sset T(K)$ under \eqref{Tlattice} generates $X_*(T)$,  we have \begin{flalign*}
		  T(K)=T(\CO_K)H_0H_1H_2.
	\end{flalign*}
	Denote by $\Delta$ the local Dynkin diagram of $G$. By \cite[\S 2.5]{tits1979reductive},  the $G^\circ(K)$-action \begin{flalign}
		G^\circ(K) \ra \Aut(\Delta)  \label{GDelta}
	\end{flalign}
	factors through the Kottwitz map $\kappa$, and the image $\Xi$ is isomorphic to $$T(K)/(T(\CO_K)C(K)H_0),$$ where $C(K)$ denotes the group of $K$-points of the center $C$ of $G$. It is straightforward to check that \begin{flalign*}
		\Xi \simeq \frac{T(\CO_K)H_0H_1H_2}{T(\CO_K)C(K)H_0}\simeq \begin{cases}
			\BZ/4\BZ\quad &\text{if $n$ is odd};\\
			\BZ/2\BZ\times \BZ/2\BZ &\text{if $n$ is even}.
		\end{cases}
	\end{flalign*}
	If $n$ is odd (resp. even), the generator(s) of $\Xi$ is $t_2$ (resp. $t_1$ and $t_2$). Set \begin{flalign*}
		\tau_1 &\coloneqq t_1(1,2n)(n,n+1)\in N_{G^\circ}(K),\\ \tau_2 &\coloneqq \begin{cases}
			t_2\sigma_2(n,n+1)\in N_{G^\circ}(K) \quad &\text{if $n$ is odd};\\ t_2\sigma_2\in N_{G^\circ}(K) &\text{if $n$ is even}.
		\end{cases} 
	\end{flalign*}
	Here, $\sigma_2\in S_{2n}^*$ is the permutation sending $(x_1,\ldots,x_{2n})$ to $(x_{n+1},\ldots,x_{2n},x_1,\ldots,x_n)$. Then $\tau_i\in \Omega$ stabilizes the base alcove $\fa$, and $\kappa(\tau_i)=\kappa(t_i)$ for $i=1,2$. Hence, $\Xi$ is also generated by (the actions of) $\tau_i$. Let us identify the vertices of $\Delta$ with the set $\tcbra{\lambda_{\cbra{j}}}_{j\in \sI}$ of self-dual lattice chains $\lambda_{\cbra{j}}$ indexed by $\sI$. The $\Xi$-action on $\Delta$ can be explicitly described in the following way. 
	\begin{enumerate}[label=(\roman*)]
		\item If $n$ is odd, then $\Xi=\pair{\tau_2}\simeq\BZ/4\BZ$. We have: $\tau_2\lambda_0=\lambda_{-n}$, $\tau_2\lambda_n=\lambda_{-0'}$, $\tau_2\lambda_{0'}=\lambda_{-n'}$, $\tau_2\lambda_{n'}=\lambda_{-0}$, and $\tau_2\lambda_i=\lambda_{-(n-i)}$ for $2\leq i\leq n-2$.
		\item If $n$ is even, then $\Xi=\pair{\tau_1,\tau_2}\simeq\BZ/2\BZ\times \BZ/2\BZ$. We have:\begin{itemize}
			\item $\tau_1\lambda_{0}=\lambda_{0'}$, $\tau_1\lambda_{0'}=\lambda_{0}$, $\tau_1\lambda_n=\lambda_{n'}$, $\tau_1\lambda_{n'}=\lambda_n$, and $\tau_1\lambda_j=\lambda_j$ for $j\in \sI\backslash\cbra{0,0',n,n'}$,
			\item $\tau_2\lambda_0=\lambda_{-n}$, $\tau_2\lambda_n=\lambda_{-0}$, $\tau_2\lambda_{0'}=\lambda_{-n'}$, $\tau_2\lambda_{n'}=\lambda_{-0'}$, and $\tau_2\lambda_i=\lambda_{-(n-i)}$ for $2\leq i\leq n-2$.
		\end{itemize} 
	\end{enumerate}
	Note that for lattices $\Lambda,\Lambda'$ in $V$ and $g\in G^\circ(K)$, if $g\Lambda=\Lambda$, then $g\Lambda^\psi=\Lambda^\psi$; and if $g\Lambda=\Lambda'$, then $g\Stab(\Lambda)g\inverse=\Stab(\Lambda')$, where $\Stab(\Lambda)=\cbra{h\in G(K)\ |\ h\Lambda=\Lambda}$.  
	By \cite[\S 2.5]{tits1979reductive}, the $G^\circ(K)$-conjugacy classes of parahoric subgroups are in bijection with the $\Xi$-orbits on the set of non-empty subsets of $\Delta$. By the above $\Xi$-action, we obtain that any maximal parahoric subgroup is $G^\circ(K)$-conjugate to some $P_{{j}}^\circ$, where $j=0$ or $2\leq j\leq \lfloor n/2\rfloor$.
\end{proof}

\begin{remark}
	It follows from the proof of Proposition \eqref{prop-paraIndex} that the $G^\circ(K)$-conjugacy classes of parahoric subgroups are in bijection with the orbit set $\CP^*(\Delta)/\Xi$, where $\CP^*(\Delta)$ denotes the set of non-empty subsets of $\Delta$.
\end{remark}

	Set \begin{flalign*}
		a_1\coloneqq \rbra{-1/2,0^{(2n-2)},1/2} \text{\ and\ } a_{n-1}\coloneqq \rbra{(-1/2)^{(n-1)},0,0,(1/2)^{(n-1)}}.
	\end{flalign*}
	Then $a_1$ is the midpoint of $a_0$ and $a_{0'}$; $a_{n-1}$ is the midpoint of $a_n$ and $a_{n'}$. Define \begin{flalign*}
		\lambda_1\coloneqq \CO_K\pair{t\inverse e_1,e_2,\ldots,e_{2n}},\ \lambda_{n-1}\coloneqq \CO_K\pair{t\inverse e_1,t\inverse e_2,\ldots,t\inverse e_{n-1},e_n,\ldots,e_{2n}}.
	\end{flalign*}
Then as in Corollary \ref{coro-stabvi}, the stabilizer subgroup of $a_1$ (resp. $a_{n-1}$) is equal to \begin{flalign*}
	  P_1\coloneqq \cbra{g\in G(K)\ |\ g\lambda_1=\lambda_1} \text{\ (resp. } P_{n-1}\coloneqq \cbra{g\in G(K)\ |\ g\lambda_{n-1}=\lambda_{n-1}}).
\end{flalign*}

For a non-empty subset $I\sset [0,n]$, set \begin{flalign*}
	P_I\coloneqq \cbra{g\in G(K)\ |\ g\lambda_i=\lambda_i, i\in I} \text{\ and\ } P_I^\circ\coloneqq P_I\cap \ker \kappa.
\end{flalign*}
We can easily reformulate Proposition \ref{prop-paraIndex} (1) as follows.
\begin{prop}\label{prop-Jform}
Let $I\sset [0,n]$ be a non-empty subset.
	The subgroup $P_I^\circ$ is a parahoric subgroup of $G^\circ(K)$. Any parahoric subgroup of $G(K)$ is $G^\circ(K)$-conjugate to a subgroup $P_I^\circ$ for some (not necessarily unique) $I\sset [0,n]$. 
	
	Moreover, the $G^\circ(K)$-conjugacy classes of maximal parahoric subgroups of $G^\circ(K)$ are in bijection with the set $\cbra{P^\circ_{j}\ |\ j=0 \text{\ or\ }2\leq j\leq \lfloor n/2\rfloor }$.
\end{prop}
\begin{remark}\label{in-i}
	By the proof of Proposition \ref{prop-paraIndex}, $P_j^\circ$ and $P_{n-j}^\circ$ are conjugate by $\tau_2\in G^\circ(K)$. Hence, $G^\circ(K)$-conjugacy classes of maximal parahoric subgroups of $G^\circ(K)$ are also in bijection with the set $\cbra{P^\circ_{j}\ |\ j=n \text{\ or\ }\lceil n/2\rceil \leq j\leq n-2 }$.
\end{remark}

The advantage of the formulation in Proposition \ref{prop-Jform} is that, for any $J\sset [0,n]$, the collection of lattices \begin{flalign}
	\lambda_J\coloneqq \cbra{\lambda_\ell}_{\ell\in 2n\BZ\pm J}  \label{eq-lambdaI}
\end{flalign}
forms a self-dual lattice chain.

\begin{defn}\label{defn-pseudomax}
	Let $I\sset [0,n]$ be a non-empty subset. We say $I$ is \dfn{pseudo-maximal} (resp. \dfn{maximal}) if $I=\cbra{i}$ for some $0\leq i\leq n$ (resp. $i=0,n$ or $2\leq i\leq n-2$). 
	
	We say a parahoric subgroup $P$ of $G(K)$ is \dfn{pseudo-maximal} if $P$ is $G^\circ(K)$-conjugate to $P_i^\circ$ (see Proposition \ref{prop-Jform})  for some $0\leq i\leq n$. 
\end{defn}

\begin{defn}
	We say a parahoric subgroup $P$ of $G(K)$ is \dfn{pseudo-maximal} if $P$ is $G^\circ(K)$-conjugate to $P_i^\circ$ (see Proposition \ref{prop-Jform})  for some $0\leq i\leq n$.
\end{defn}

By Proposition \ref{prop-Jform}, a pseudo-maximal parahoric subgroup $P$ of $G(K)$ is maximal if and only if $P$ is not $G^\circ(K)$-conjugate to $P_1$ or $P_{n-1}$.

\begin{remark}
	In the literature (e.g. \cite{pappas2022regular,howard2017rapoport}), people also use vertex lattices to describe pseudo-maximal parahoric subgroups. In our setting, the standard lattice $\Lambda_i\sset V$ gives a vertex lattice such that $\pi\Lambda_i\subset \Lambda_i^\psi\sset \Lambda_i$ in $V$ and $\dim_k(\Lambda_i/\Lambda_i^\psi)=2i$.  
\end{remark}

\section{Naive and spin local models} \label{sec-naivespinmod}
In this section, we first review the constructions of naive and spin local models following \cite{rapoport1996period,pappas2009local,smithling2011topological}. We then identify their special fibers with unions of Schubert cells in an affine flag variety. These Schubert cells are indexed by certain permissible subsets in the Iwahori-Weyl groups.

\subsection{}  \label{subsec-defnmodels}
In this subsection, we work over the field $F$ in \S \ref{subsec-notation-intro}. As in \S \ref{sec-para}, we have the split orthogonal similitude group $G\coloneqq \GO(V,\psi)$ attached to a symmetric space $(V,\psi)$ of $F$-dimension $2n$, which we assume split, with distinguished basis $e_1,\ldots,e_{2n}$. In this case, we will denote the standard lattice chain \eqref{eq-lambdaI} by $\Lambda_I$ for a non-empty subset $I\sset [0,n]$. 

Let $\CL$ be a self-dual periodic lattice chain of $V$. For each $\Lambda\in\CL$, we let $$\Lambda^\psi\coloneqq \cbra{x\in V\ |\ \psi(x,\Lambda)\sset \CO } $$ denote the dual lattice of $\Lambda$ with respect to $\psi$.  

\begin{defn}[{cf. \cite{rapoport1996period,smithling2011topological}}] \label{defn-naive}
    The naive local model $\RM^\naive_\CL$ is the functor \[\Sch_{/\CO}^\op\ra \Sets \] sending a scheme $S$ over $\CO$ to the set of $\CO_S$-modules $(\CF_\Lambda)_{\Lambda\in \CL}$ such that 
    \begin{itemize}
        \item [LM1.] for any $\Lambda\in \CL$, Zariski locally on $S$, $\CF_\Lambda$ is a direct summand of $\Lambda_S\coloneqq \Lambda\otimes_\CO\CO_S$ of rank $n$;
        \item[LM2.] for any $\Lambda\in \CL$, the perfect pairing \[\psi\otimes 1\colon \Lambda_S\times \Lambda^\psi_S \ra \CO_S \] induced by $\psi$ satisfies $(\psi\otimes 1)(\CF_\Lambda, \CF_{\Lambda^\psi})=0$;
        \item[LM3.] for any inclusion $\Lambda\sset \Lambda'$ in $\CL$, the natural map $\Lambda_S\ra \Lambda'_S$ induced by $\Lambda\hookrightarrow \Lambda'$ sends $\CF_\Lambda$ to $\CF_{\Lambda'}$; and the isomorphism $\Lambda_S\simto (\pi\Lambda)_S$ induced by $\Lambda\overset{\pi}{\ra}\pi\Lambda$ identifies $\CF_\Lambda$ with $\CF_{\pi\Lambda}$.
    \end{itemize}
\end{defn}
The functor $\RM^\naive_\CL$ is representable by a projective scheme over $\CO$, which we will also denote by $\RM^\naive_\CL$. We have \begin{equation*}
    \RM^\naive_{\CL,F} \simeq \OGr(n,V), 
\end{equation*} where $\OGr(n,V)$ denotes the orthogonal Grassmannian of (maximal) totally isotropic $n$-dimensional subspaces of $V$. In particular, $\RM^\naive_{\CL,F}$
has dimension $n(n-1)/2$, and consists of two connected components $\OGr(n,V)_\pm$, which are isomorphic to each other. We can identify $\OGr(n,V)_\pm$ with the flag variety $G^\circ/P_{\mu_\pm}$, where $P_{\mu_\pm}$ denotes the parabolic subgroup associated to a minuscule cocharacter $\mu_\pm$ given by \begin{flalign}
    \mu_+\coloneqq (1^{(n)},0^{(n)}) \text{\ and\ } \mu_-\coloneqq (1^{(n-1)},0,1,0^{(n-1)}). \label{cochar}
\end{flalign}

In \cite[\S 7]{pappas2009local}, Pappas and Rapoport defined an involution operator $a$ on $\wedge^{n}_FV$ inducing a decomposition  \begin{flalign}
	\wedge^{n}_FV= W_+\oplus W_-, \label{Wdecomp}
\end{flalign} where $W_\pm$ denotes the $\pm 1$-eigenspace for $a$. For an $\CO$-lattice $\Lambda$ in $V$, set \[(\wedge^{n}_\CO\Lambda)_\pm \coloneqq \wedge^{n}_\CO\Lambda\cap W_\pm.  \]
Then we define a refinement of $\RM^\naive_\CL$ as follows, see \cite[\S 2.3]{smithling2011topological}.

% we can also use the command \sideset{}{_\CO^{n}}{\bigwedge} \Lambda, if we want a big wedge symbol and put the upper script the right-hand side of the symbol wedge.

\begin{defn}\label{defn-spin}
    The spin local model $\RM^\pm_\CL$ is the functor \[\Sch_{/\CO}^\op\ra \Sets \] sending a scheme $S$ over $\CO$ to the set of $\CO_S$-modules $(\CF_\Lambda)_{\Lambda\in \CL}$, which satisfies LM1-3 and the condition 
    \begin{itemize}
        \item[LM4$\pm$.] for any $\Lambda\in\CL$, Zariski locally on $S$,  the line $\bigwedge_{\CO_{S}}^{n} \mathcal{F}_{\Lambda}$ is contained in
    \[
    \operatorname{Im}\left[ \left( \wedge_{\CO}^{n} \Lambda \right)_{\pm} \otimes_{\CO} \CO_{S} \longrightarrow \wedge_{\CO_{S}}^{n} \Lambda_{S} \right].
    \]
    \end{itemize}
\end{defn}

The functor $\RM^\pm_\CL$ is representable by a closed subscheme of $\RM^\naive_\CL$ over $\CO$, which we will also denote by $\RM^\pm_\CL$. Moreover, we have (cf. \cite[8.2.1]{pappas2009local}) \[\RM^\pm_{\CL,F}\simeq \OGr(n,V)_\pm. \]

\begin{remark}\label{isomconjugate}
	Let $\CL_1, \CL_2$ be two self-dual lattice chains of $V$. If $g\CL_1=\CL_2$ for some $g\in G^\circ(K)$, then clearly $g$ defines isomorphisms \begin{flalign*}
		\RM^\naive_{\CL_1}\simto \RM^\naive_{\CL_2} \text{\ and\ } \RM^\pm_{\CL_1}\simto \RM^\pm_{\CL_2}.
	\end{flalign*}
\end{remark}

Denote by \begin{flalign}
	\sG\coloneqq \sG_{\CL}  \label{Gpara}
\end{flalign}  
the (affine smooth) group scheme of similitude automorphisms of $\CL$.  By Bruhat--Tits theory, the neutral component $\sG^\circ$ is a parahoric group scheme for $G^\circ$.

\begin{lemma}\label{lemstable}
	The parahoric group scheme $\sG^\circ\sset \sG$ preserves the subspace $\RM^\pm_\CL\sset \RM^\naive_\CL$. 
\end{lemma}
\begin{proof}
	By \cite[\S 7.1.3]{pappas2009local}, the decomposition \eqref{Wdecomp} is preserved by $G^\circ$. Since the generic fiber of $\sG^\circ$ is $G^\circ$, the space $(\wedge^n_\CO\Lambda)_\pm$ in LM4$\pm$ for $\Lambda\in\CL$ is stable under the action of $\sG^\circ$. In particular, condition LM4$\pm$ is $\sG^\circ$-stable, and $\sG^\circ$ preserves the subspace $\RM^\pm_\CL$.
\end{proof}

\begin{defn}\label{Mlocdefin}
    Denote by $\RM^{\pm\loc}_\CL$ the schematic closure of $\RM^\pm_{\CL,F}$ in $\RM^\pm_\CL$.
\end{defn}
Recall that given $\sG^\circ$ and the minuscule cocharacter $\mu_\pm$ in \eqref{cochar}, the associated \dfn{schematic local model} $\RM^\loc_{\sG^\circ,\mu_\pm}$ represents the corresponding v-sheaf local model in the sense of Scholze--Weinstein \cite[\S 21.4]{scholze2020berkeley}; see \cite[Theorem 7.21]{anschutz2022p}.
 
\begin{prop}\label{propiso}
	The scheme $\RM^{\pm\loc}_\CL$ is isomorphic to $\RM^\loc_{\sG^\circ,\mu_\pm}$. In particular, $\RM^{\pm\loc}_\CL$ is $\CO$-flat of relative dimension $n(n-1)/2$, normal, and Cohen-Macaulay with reduced special fiber. 
\end{prop}
\begin{proof}
	By results in \cite[\S 8.2.5]{pappas2013local}, $\RM^{\pm\loc}_\CL$ is isomorphic to the Pappas--Zhu local model attached to the parahoric group scheme $\sG^\circ$ and the cocharacter $\mu_\pm$. In particular, the scheme $\RM^{\pm\loc}_\CL$ is $\CO$-flat, normal, and Cohen-Macaulay with reduced special fiber by \cite{pappas2013local,haines2022normality}. On the other hand, it is well-known that the Pappas--Zhu model is isomorphic to the corresponding schematic local model (see, e.g., \cite[Theorem 2.15]{he2020good}). Hence, we obtain that $\RM^{\pm\loc}_\CL\simeq \RM^\loc_{\sG^\circ,\mu_\pm}$.
\end{proof}

\begin{conj}[{\cite[Conjecture 8.1]{pappas2009local}}] \label{introconj-pr}
    The spin local model $\RM^\pm_\CL$ is flat over $\CO$. Equivalently, $\RM^\pm_\CL=\RM^{\pm\loc}_\CL$.
\end{conj}

\begin{remark} \label{rmk-toI}
	By Proposition \ref{prop-Jform}, for any self-dual lattice chain $\CL$, there exists an element $g\in G^\circ(K)$ and a non-empty $I\sset[0,n]$ such that $\CL=g\Lambda_I$. By Remark \ref{isomconjugate}, we have isomorphisms \begin{flalign*}
	\RM^\naive_\CL\simeq \RM^\naive_{\Lambda_I} \text{\ and\ } \RM^\pm_\CL\simto \RM^\pm_{\Lambda_I}.
\end{flalign*}  Hence, to prove Theorem \ref{intro-thmtopo}, it suffices to treat the cases $\CL=\Lambda_I$. In addition, by Remark \ref{in-i}, we have  $\RM^\naive_{\Lambda_{\cbra{i}}}\simeq\RM^\naive_{\Lambda_{\cbra{n-i}}}$ and $\RM^\pm_{\Lambda_{\cbra{i}}}\simeq\RM^\pm_{\Lambda_{\cbra{n-i}}}$ for $0\leq i\leq n$. 
\end{remark}

In the rest of the paper, we let $\RM^\naive_{I,k}$ (resp. $\RM^\pm_{I,k}$) denote the special fiber of $\RM^\naive_{\Lambda_I}$ (resp. $\RM^\pm_{\Lambda_I}$). Sometimes, we simply write $\RM^\naive_k$ or $\RM^\pm_k$ when the index set is clear. 

\subsection{Embedding \texorpdfstring{$\RM^\naive_{I,k}$}{$\RM_{I,k}^{\text{naive}}$} in  the affine flag variety}
From now on, we take the field $K$ in \S \ref{sec-para} to be the field $k((t))$ of Laurent series over $k$. For ease of notation, we use the same symbol $G$ for the split orthogonal similitude group over $F$ or $K$; the base field will be clear from context.

Let $I$ be a non-empty subset of $[0,n]$.
For the associated self-dual lattice chain $\lambda_I$ (see \eqref{eq-lambdaI}) of $K^{2n}$, we let $\CP_I$ be the $\CO_K$-group scheme of similitude automorphisms of $\lambda_I$. 
As $p\neq 2$, the group scheme $\CP_I$ is affine and smooth over $\CO_K$ by \cite[Theorem 3.16]{rapoport1996period}, and the generic fiber of $\CP_I$ is $G$. The neutral component $\CP^{\circ}_I$ of $\CP_I$ is the parahoric group scheme attached to $\lambda_I$. Denote by $\CQ_I$ the schematic closure of $G^{\circ}$ in $\CP_I$. Set $Q_I\coloneqq \CQ_I(\CO_K)$. We have \begin{flalign*}
	P_I=\CP_I(\CO_K), Q_I=P_I\cap G^\circ(K), \text{\ and\ } P^0_I=\CP^{\circ}_I(\CO_K).
\end{flalign*} 

\begin{defn}[{\cite[\S 6]{smithling2011topological}}]  \label{defnFlag}
	The affine flag variety $\sFl_I$ is the fpqc quotient sheaf\footnote{Although $\CP_I$ need not equal $\CP_I^\circ$ in general, we still refer to $\sFl_I$ as an affine flag variety. } $$\sFl_I\coloneqq LG/L^+\CP_I.$$ Taking $\lambda_I$ as the base point, one identifies $\sFl_I$ with the fpqc sheaf sending a $k$-algebra $R$ to the set of $R[[t]]$-lattice chains $(L_{\ell})_{\ell\in 2n\BZ\pm I}$ in $R((t))^{2n}$ with the property that \begin{enumerate}
	\item for $\ell<\ell'$ in $2n\BZ\pm I$, we have $L_\ell\sset L_{\ell'}$, and the quotient $L_{\ell'}/L_\ell$ is locally free of $R$-rank $\dim_k\lambda_{\ell'}/\lambda_\ell$;
	\item if $\ell'=\ell+2nd$ for some $d\in\BZ$, then $L_{\ell'}=t^{-d}L_\ell$;
	\item Zariski-locally on $\Spec R$, there exists a scalar $\alpha\in R((t))\cross$ such that ${L}_\ell^{\psi}=\alpha L_{-\ell}$ for all $\ell\in 2n\BZ\pm I$. Here ${L}_\ell^{\psi}\coloneqq \cbra{x\in R((t))^{2n}\ |\ \psi(L_\ell,x)\sset R[[t]] }. $
\end{enumerate} 
\end{defn}

\begin{lemma} \label{lem-Itype}
	Let $I\sset [0,n]$ be a non-empty subset.
	\begin{enumerate}
		\item If $\cbra{0,n}\sset I$, then $P_I=Q_I=P_I^\circ$.
		\item If $0\in I$ and $n\notin I$, then $P_I/Q_I\simeq\BZ/2\BZ$ is generated by $\tau$; and $Q_I=P_I^\circ$.
		\item If $0\notin I$ and $n\in I$, then $P_I/Q_I\simeq\BZ/2\BZ$ is generated by $$\tau'\coloneqq \diag(t\inverse,1,\ldots,1,t)(1,2n);$$ and $Q_I=P_I^\circ$.
		\item If $0,n\notin I$, then $P_I/Q_I\simeq \BZ/2\BZ$ is generated by $\tau$; and $Q_I/P_I^\circ\simeq\BZ/2\BZ$ is generated by $\tau_1=\diag(t\inverse,1,\ldots,1,t)(1,2n)(n,n+1)$.
	\end{enumerate}
\end{lemma}
\begin{proof}
Note that $\cbra{0}$ or $\cbra{n}$ corresponds to a hyperspecial subgroup. Hence, we have $Q_0=P_0^\circ\sset\ker \kappa$ and $Q_n=P_n^\circ\sset \ker\kappa$; cf. \cite[Proposition 8.4.14]{kaletha2023bruhat} and \cite[Proposition 3]{haines2008parahoric}.

	(1) Suppose that $g\in P_I-Q_I$. By \eqref{Gtwocompo}, there exists a $g'\in G^\circ(K)$ such that $g=\tau g'$. Since $\cbra{0,n}\sset I$, we have $g\Lambda_0=\tau g'\Lambda_0=\Lambda_0$ and $g\Lambda_n=\tau g'\Lambda_n=\Lambda_n$. Thus, \begin{flalign}
		   g'\Lambda_0=\tau\Lambda_0=\Lambda_0 \text{\ and\ } g'\Lambda_n=\tau\Lambda_n=\Lambda_{n'}.  \label{eq34}
	\end{flalign}
	Denote by $B$ the Iwahori subgroup (poitwisely) stabilizing the alcove $\fa$. 
	By the affine Bruhat decomposition (cf. \cite[Proposition 5.1.1]{kaletha2023bruhat}) for $G^\circ$, we can write $$g'=b_1wb_2,$$ where $b_1,b_2\in B$ and $w\in N_{G^\circ}(K)$.  We have $b_i\Lambda_0=\Lambda_0$ and $b_i\Lambda_{n'}=\Lambda_{n'}$ for $i=1,2$. Hence, \eqref{eq34} amounts to \begin{flalign*}
		   w\Lambda_0=\Lambda_0 \text{\ and\ } w\Lambda_n=\Lambda_{n'}.
	\end{flalign*}
	Note that we can write $w=t\sigma$, where $t\in T(K)$ and $\sigma\in S_{2n}^\circ$. The equality $w\Lambda_0=\Lambda_0$ implies that $t\in T(\CO_K)$. Then the equality $w\Lambda_n=\Lambda_{n'}$ implies that $\sigma\Lambda_n=\Lambda_{n'}$. However, this is impossible, as $\sigma\in S_{2n}^\circ$ is an even permutation. 
	Hence, we obtain that $P_I=Q_I$. Since $0\in I$, we have $Q_I\sset Q_0\sset \ker\kappa$. So $P_I=Q_I=P_I^\circ$.
	
	(2) It is clear that $\tau\in P_I-Q_I$ by assumption. Since $$P_I/Q_I\hookrightarrow G(K)/G^\circ(K)=\BZ/2\BZ,$$ we have $P_I/Q_I\simeq \BZ/2\BZ$. In addition, $Q_I=P_I^\circ$ as $Q_I\sset Q_0\sset \ker\kappa$. 
	
	(3) The proof is similar as that of (2).
	
	(4) The first part follows as in the proof of (2). Let $g\in Q_I$. Then $g\Lambda_i=\Lambda_i$ for any $i\in I$; and hence $\det g\in \CO\cross_K$ and $c(g)\in \CO_K\cross$. The Kottwitz homomorphism \eqref{eq-kott} induces a map \begin{flalign*}
		   Q_I/P_I^\circ\hookrightarrow \cbra{0}\times\BZ/2\BZ\sset \BZ\times \BZ/2\BZ.
	\end{flalign*}
	We thus obtain that $Q_I/P_I^\circ\simeq\BZ/2\BZ$.
\end{proof}

\begin{defn}
	We say that a subset $I\sset [0,n]$ is of {type \RNum{1}} if $\cbra{0,n}\sset I$; of type \RNum{2} if exactly one of $0$ or $n$ lies in $I$; and of type \RNum{3} if $0,n\notin I$. 
\end{defn}

By Lemma \ref{lem-Itype}, we have  \begin{flalign} \label{sFI}
    \sFl_I= LG/L^+\CP_I \simeq \begin{cases}
    	(LG^{\circ}/L^+\CP_I^\circ) \sqcup \tau(LG^\circ/L^+\CP_I^\circ) \ &\text{if $I$ is of type \RNum{1}};\\ 
    	LG^\circ/L^+\CQ_I\simeq LG^\circ/L^+\CP_I^\circ &\text{if $I$ is of type \RNum{2}}; \\
    	LG^\circ/L^+\CQ_I &\text{if $I$ is of type \RNum{3}}.
    \end{cases}  
\end{flalign}
Set $\sFl'_I\coloneqq LG^{\circ}/L^+\CP^{\circ}_I$ and $H\coloneqq \ker(LG^{\circ}\xrightarrow{\kappa} \BZ\oplus\BZ/2\BZ\twoheadrightarrow \BZ/2\BZ)$. Here, we use the fact that the Kottwitz homomorphism $\kappa\colon G^{\circ}(K)\ra \BZ\oplus\BZ/2\BZ$ extends to a homomorphism $LG^{\circ}\ra \BZ\oplus\BZ/2\BZ$ classifying the connected components of $LG^\circ$, cf. \cite[Theorem 0.1]{PR08}.  If $I$ is of type \RNum{3}, then (cf. \cite[6.2.3]{smithling2014topological}) 
    \begin{flalign}
    	\sFl_I'= (H/L^+\CP_I^\circ)\sqcup (\tau H/L^+\CP_I^\circ) \simeq \sFl_I\sqcup \sFl_I   \label{HsFI}
    \end{flalign}
is isomorphic to two copies of $\sFl_I$. 

\begin{defn}
    For $w\in\wt{W}$, the Schubert cell $C_w$ is the reduced $k$-subscheme  \[C_w\coloneqq L^+\CP^{\circ}_Iw \sset\sFl_I, \]
    and the Schubert variety $S_w$ is the reduced closure of $C_w$ in $\sFl_I$. 
\end{defn}

\begin{prop} \label{lem-cell}
	There is a bijection between the set of Schubert cells in $\sFl_I$ and the double coset \begin{flalign*}
		 \begin{cases}
		 	W_I\backslash\wt{W}/W_I \quad &\text{if $I$ is of type \RNum{1}},\\ W_I\backslash\wt{W}^\circ/W_I &\text{if $I$ is of type \RNum{2}},\\ W_I\backslash W'/W_I &\text{if $I$ is of type \RNum{3}},
		 \end{cases}  
	\end{flalign*}
	where $W'\sset \wt{W}^\circ$ is defined in \eqref{Wprime}, and $W_I\coloneqq \cbra{w\in W_\aff\ |\ wa_i=a_i \text{\ for\ } i\in I}$.
\end{prop}
\begin{proof}
    Suppose that $I$ is of type \RNum{1}. Then $\sFl_I=LG/L^+\CP_I^\circ$. Set \begin{flalign*}
    	I'\coloneqq \begin{cases}
    		I \quad &\text{if $n\notin I$},\\ (I\backslash\cbra{n})\cup\cbra{n'} &\text{if $n\in I$}.
    	\end{cases}
    \end{flalign*}
    Define $$W_{I'}\coloneqq \cbra{w\in W_\aff\ |\ wa_i=a_i, i\in I'}.$$ Then $\tau W_I\tau =W_{I'}$.
	By the standard affine Bruhat decomposition for connected reductive groups (see  \cite[Proposition 4.8, pp. 184]{PRS13}), we have bijections \begin{flalign}
		    W_I\backslash\wt{W}^\circ/W_I\simto P_I^\circ\backslash G^\circ(K)/P_I^\circ \text{\ and\ } W_{I'}\backslash\wt{W}^\circ/W_I\simto P_{I'}^\circ\backslash G^\circ(K)/P_I^\circ.   \label{eqWIPI}
	\end{flalign}
	Since $\wt{W}=\wt{W}^\circ\sqcup \tau\wt{W}^\circ$, we obtain that \begin{flalign*}
		W_I\backslash \wt{W}/W_I &\simeq (W_I\backslash \wt{W}^\circ/W_I)\sqcup (\tau W_{I'}\backslash\wt{W}^\circ/W_I)\\ &\simeq (P_I^\circ\backslash G^\circ(K)/P_I^\circ)\sqcup (\tau P_{I'}^\circ\backslash G^\circ(K)/P_I^\circ)\\ &\simeq P_I^\circ\backslash G(K)/P_I^\circ.
	\end{flalign*}
	This proves the case where $I$ is of type \RNum{1}.
	
	When $I$ is of type \RNum{2}, the lemma follows from \eqref{sFI} and \eqref{eqWIPI}.
	
	Suppose that $I$ is of type \RNum{3}. Then $\sFl_I$ is isomorphic to $H/L^+\CP_I^\circ$ by \eqref{HsFI}. The lemma follows from \eqref{eqWIPI} and the definition of $W'$. 
\end{proof}

\begin{lem}\label{lem-Wmperm}
    Suppose that $I=\cbra{i}\sset [0,n]$. Denote $W_i\coloneqq W_I$.
    We have $$W_i\simeq S_{2i}^\circ\times S_{2(n-i)}^\circ\sset S_{2n}^\circ,$$ where the two factors act as even permutations on $A_i\coloneqq \sbra{1,i}\cup [i^*,2n]$ and $B_i\coloneqq \sbra{i+1,i^*-1}$ respectively.
\end{lem}
\begin{proof}
    Let $w=t^ww_0\in W_i\sset W_\aff$. By definition, we have $wa_i=a_i$. Namely,  \begin{flalign}
    	a_i(w_0\inverse(j))+t^w(j)=a_i(j) \text{\ for all $1\leq j\leq 2n$, and $\varepsilon(t^w)=0$}.  \label{am}
    \end{flalign}
    Recall that $\varepsilon$ (see Definition \ref{defnWprime}) is given by $\varepsilon(t^w)=\sum_{j=1}^nt^w(j)\mod 2$.
    By \eqref{am}, the translation part \begin{equation}
    	t^w=a_i-w_0a_i \label{tw}
    \end{equation} is determined by $w_0$. 
    Since $w_0\in S_{2n}^\circ$ and $a_i-w_0a_i=t^w\in\BZ^{2n}$, we conclude that $w_0$ permutes the subsets $A_i$ and $B_i$.  
    Write $\sigma$ for the restriction of $w_0\inverse$ on $A_i$. Then $\sigma$ is a composition of permutations of the form $\tau_{j'j}\coloneqq (j'j)(j'^*j^*)$ and $\tau_j\coloneqq (jj^*)$, where $j,j'\in A_i$ and $j'\neq j^*$. Let $r(\sigma)$ be the number of transpositions of the form $\tau_i$ occurring in a (reduced) decomposition of $\sigma$ into transpositions. Then $\sigma$ is an even permutation if and only if $r(\sigma)$ is even. Recall that $a_i=((-1/2)^{(i)},0^{(2n-2i)},(1/2)^{(i)})$. By \eqref{tw}, we have \begin{flalign*}
        \varepsilon(t^w) &=\sum_{k=1}^n(a_i(k)-a_i(w_0\inverse(k))) \mod 2  \\ &=\rbra{-(1/2)i-\sum_{k=1}^i a_i(w_0\inverse(k))-\sum_{k=i+1}^n a_i(w_0\inverse(k))}\mod 2\\ &=\rbra{-(1/2)i-\sum_{k=1}^i a_i(\sigma(k))} \mod 2 \\ &=\rbra{-(1/2)i-(-(1/2)i+r(\sigma))} \mod 2\\ &= -r(\sigma) \mod 2.
    \end{flalign*}
    Since $\varepsilon(t^w)=0$, we obtain that $r(\sigma)$ is even; hence $\sigma$ is an even permutation.
    As $w_0\inverse\in S_{2n}^\circ$ is even, we obtain that 
        \text{$w_0$ restricting to $A_i$ (resp. $B_i$) is even. }
   Therefore, the projection $\wt{W}^\circ\twoheadrightarrow S_{2n}^\circ$ induces a homomorphism $$f\colon W_i\ra  S_{2i}^\circ\times S_{2(n-i)}^\circ. $$
    By \eqref{tw}, the map $f$ is an injection. For any $w_0\in S_{2i}^\circ\times S_{2(n-i)}^\circ$, the equation \eqref{tw} defines an element $t^ww_0$ in $W_i$. It follows that $f$ is also surjective; hence $W_i\simeq S_{2i}^\circ\times S_{2(n-i)}^\circ$. 
\end{proof}

Using the obvious isomorphism \[\CO/\pi\CO\simeq k\simeq \CO_K/t\CO_K, \] we may identify \[\Lambda_I\otimes_\CO k\simeq \lambda_I\otimes_{\CO_K}k. \] 
This induces isomorphisms \begin{flalign}
	\sG_I\otimes_\CO k\simeq \CP_I\otimes_{\CO_K}k \text{\ and\ } \sG^\circ_I\otimes_\CO k\simeq \CP_I^{\circ}\otimes_{\CO_K}k.  \label{sGisom}
\end{flalign}
Let $R$ be a $k$-algebra and let $(\CF_{\Lambda_i})_{\Lambda_i\in\Lambda_I}\in \RM^\naive_I(R)$ (see \S \ref{subsec-defnmodels}). Denote by  $$\wt{\CF}_{\Lambda_i}\sset \lambda_i\otimes_{\CO_K}R[[t]] $$ the inverse image of the $R$-submodule $\CF_{\Lambda_i}\sset \Lambda_i\otimes_\CO R\simeq\lambda_i\otimes_kR$ along the natural reduction map $\lambda_i\otimes_{\CO_K}R[[t]]\twoheadrightarrow \lambda_i\otimes_kR$. In this way, we obtain a closed immersion (we may globally take the scalar $\alpha=t\inverse$ in Definition \ref{defnFlag} (3), cf. \cite[\S 7.1]{smithling2011topological}) \begin{flalign*}
    \RM^\naive_{I,k}\hookrightarrow \sFl_I,\quad (\CF_{\Lambda_i})\mapsto (\wt{\CF}_{\Lambda_i}).
\end{flalign*}
By construction, the action of $L^+\CP_I$ on $\sFl_I$ preserves the closed subscheme $\RM^\naive_{I,k}$. By Lemma \ref{lemstable} and \eqref{sGisom}, the $L^+\CP^{\circ}_I$-action preserves $\RM^\pm_{I,k}$. Therefore, the underlying topological spaces of $\RM^\naive_{I,k}$ and $\RM^\pm_{I,k}$ are unions of Schubert cells in $\sFl_I$.

\begin{defn}\label{defn-permi}
    Fix a non-empty subset $I\sset [0,n]$. 
    For $w\in\wt{W}$, we say that $w$ is \dfn{naively-permissible} (resp. \dfn{$\pm$-permissible})  if the Schubert cell $C_w$ is contained in $\RM^\naive_{I,k}$ (resp. $\RM^\pm_{I,k}$).
\end{defn}

It is sometimes convenient to consider the following equivalent definition of $\RM^\naive_I$ (see \cite[\S 2]{smithling2011topological}). Denote by $(\varepsilon_i)_{i=1}^{2n}$ the standard basis of $\CO^{2n}$. For $1\leq i\leq 2n$, let $f_i\colon\CO^{2n}\ra \CO^{2n}$ denote the $\CO$-linear map sending $\varepsilon_i$ to $\pi\varepsilon_i$, and sending $\varepsilon_j$ to $\varepsilon_j$ for $j\neq i$. Then there exists a unique isomorphism of chains of $\CO$-modules \begin{equation} \label{fi}
    \begin{split}
    	\xymatrix{ \cdots\ar@{^{(}->}[r] &\Lambda_0\ar[d]^{\simeq}\ar@{^{(}->}[r] &\Lambda_1\ar@{^{(}->}[r]\ar[d]^{\simeq} &\cdots \ar@{^{(}->}[r] &\Lambda_{2n}\ar[d]^{\simeq}\ar@{^{(}->}[r] &\cdots \\ \cdots\ar[r]^{f_{2n}} &\CO^{2n}\ar[r]^{f_1} &\CO^{2n}\ar[r]^{f_2} &\cdots\ar[r]^{f_{2n}} &\CO^{2n}\ar[r]^{f_1} &\cdots
	 }
    \end{split} 
\end{equation}
such that the leftmost vertical arrow identifies the standard ordered basis $e_1,e_2,\ldots,e_{2n}$ of $\Lambda_0$ with the basis $\varepsilon_1,\ldots,\varepsilon_{2n}$ of $\CO^{2n}$. We equip $\CO^{2n}$ with the perfect split symmetric $\CO$-bilinear pairing whose corresponding matrix is $H_{2n}$ with respect to the basis $(\varepsilon_i)_{1\leq i\leq 2n}$.  Write \begin{flalign*}
	  I=\cbra{i_0<i_1<\cdots<i_r}\sset [0,n].
\end{flalign*} For an $\CO$-algebra $R$, the set $\RM^\naive_{I}(R)$ of $R$-points consists of $R$-submodules $\CF_{j}$ of $R^{2n}$ for $j\in 2n\BZ\pm I$ such that 
\begin{itemize}
	\item[LM$1'$.] each $\CF_j$ is a locally direct summand of rank $n$ of $\Lambda_j\otimes  R$;
	\item[LM$2'$.] $(f_{ij}\otimes R)(\CF_{i})\sset \CF_{j}$ for $i<j$ in $2n\BZ\pm I$, where $f_{ij}$ denotes the composition of maps $f_k$ corresponding to the inclusion $\Lambda_i\hookrightarrow \Lambda_j$ in \eqref{fi};
	\item[LM$3'$.] for $j\in 2n\BZ\pm I$, we have $\CF_j=\CF_{2n+j}$ and $\CF^\perp_j=\CF_{2n-j}$. Here, $\CF_j^\perp$ denotes the orthogonal complement of $\CF_j$ in $R^{2n}$ with respect to the split symmetric pairing on $R^{2n}$.
\end{itemize}

\subsection{Schubert cells in $\RM^\naive_{I,k}$ and $\RM^\pm_{I,k}$}
Fix a non-empty subset $I\sset [0,n]$.
\begin{defn}\label{defn-face}
    An \dfn{$I$-face} (or simply \dfn{face}) is a family $(v_j)_{j\in 2n\BZ\pm I}$ of vectors in $\BZ^{2n}$ such that for all $i,j\in 2n\BZ\pm I$, we have \begin{enumerate}
        \item $v_{i+2n}=v_i-\mathbf 1$;
        \item $v_i\geq v_j$ if $i\leq j$; 
        \item $\Sigma v_i-\Sigma v_{j}=j-i$;
        \item there exists $d\in\BZ$ such that $v_j+v_{-j}^*=\mathbf d$. 
    \end{enumerate}
\end{defn}
Note that the action of $\wt{W}$ on $\BR^{2n}$ (see \eqref{actionW}) induces an action on the set of faces.
\begin{defn}
	For $j=2nd+i$ with $d\in\BZ$ and $0\leq i\leq 2n$, set \begin{flalign*}
		\omega_j\coloneqq ((-1)^{(i)},0^{(2n-i)})-\mathbf d \in \BZ^{2n}.
	\end{flalign*} 
\end{defn} 
The family $(\omega_j)_{j\in 2n\BZ\pm I}$ is clearly a face.

Let $R$ be a $k$-algebra. By construction, the image of the embedding $\RM^\naive_{I,k}(R)\hookrightarrow \sFl_I(R)$ consists of $(L_j)_{j\in 2n\BZ\pm I}$ in $\sFl_I(R)$ such that, for $j=2n\BZ\pm I$, \begin{itemize}
    \item[(M1)] $\lambda_{j}\otimes_{\CO_K} R[[t]]\supset L_j\supset t\lambda_{j}\otimes_{\CO_K}R[[t]]$, and
    \item[(M2)] the $R$-module $(\lambda_j\otimes_{\CO_K}R[[t]])/L_j$ is locally free of rank $n$. 
\end{itemize}

\begin{lemma}\label{permp1p2}
	Let $w\in\wt{W}$. Then $w$ is naively-permissible (see Definition \ref{defn-permi}) if and only if for any $d\in\BZ$ and $0\leq i\leq 2n$ with $j=2nd+i\in 2n\BZ\pm I$, we have \begin{enumerate}
    \item[(P1)] $\omega_j\leq w\omega_j\leq \omega_j+\mathbf 1$;
    \item[(P2)] $\Sigma (w\omega_j)=n-i-2nd$.
    \end{enumerate}
\end{lemma}
\begin{proof}
	(cf. \cite[\S 7.3]{smithling2011topological}) Note that $w$ corresponds to $(w\lambda_j)_{j\in 2n\BZ\pm I}\in\sFl_I(k)$. If $w\omega_j=(r_1,\ldots,r_{2n})\in\BZ^{2n}$, then a direct computation gives  \begin{flalign}
		  w\lambda_j=k[[t]]\pair{t^{r_1}e_1,\ldots,t^{r_{2n}}e_{2n}}. \label{ai}
	\end{flalign} 
	Note that $w$ is naively-permissible if and only if $w\lambda_j$ satisfies (M1) and (M2). By \eqref{ai}, we obtain that (M1) (resp. (M2)) is equivalent to (P1) (resp. (P2)). This proves the lemma.
\end{proof}

%Let $(f_i)_{1\leq i\leq 2n}$ and $(g_i)_{1\leq i\leq 2n}$ denote the following $\CO$-bases of $\Lambda_m$ and $\Lambda_{m}^\psi$, respectively:
%\begin{flalign*}
%    \cbra{\pi\inverse e_1,\ldots,\pi\inverse e_m,e_{m+1},\ldots,e_{2n} }\ \text{and\ } \cbra{e_1,\ldots,e_{3m},\pi e_{3m+1},\ldots,\pi e_{2n}}.
%\end{flalign*} 
%For an $\CO$-algebra $R$, we use the same notation for the induced $R$-bases of $\Lambda_m\otimes R$ and $\Lambda_{m}^\psi\otimes R$.

\begin{lem} \label{lem-perm}
    Let $w\in\wt{W}$ be naively-permissible. Let $j\in 2n\BZ\pm I$. \begin{enumerate}
        \item Set $v_j\coloneqq w\omega_j$. We have $v_j+v_{-j}^*=\mathbf 1$. In particular, $v_{-j}$ is completely determined by $v_j$. 
        \item Set $\mu^w_{j}\coloneqq w\omega_{j}-\omega_{j}$. The point $(w\lambda_j)_{j\in 2n\BZ\pm I}\in \sFl_I(k)$ corresponds to $(\CF^w_j)_{j\in 2n\BZ\pm I}\in \RM^\naive_I(k)$, where \begin{flalign*}
        \CF^w_j&=k\pair{\varepsilon_i\ |\ \mu_j^w(i)=0 \text{\ and\ } 1\leq i\leq 2n}.
    \end{flalign*}
    \end{enumerate} 
\end{lem}
\begin{proof}
    (1) Since $(v_j)_{j\in 2n\BZ\pm I}$ is a face, there exists a constant $d\in\BZ$ such that $$d=v_j(i)+v_{-j}(i^*)$$ for all $1\leq i\leq 2n$. By condition (P1), we have  \begin{flalign*}
        \omega_{j}(i)+\omega_{-j}(i^*)\leq v_j(i)+v_{-j}(i^*)\leq \omega_{j}(i)+\omega_{-j}(i^*)+2.
    \end{flalign*}
    Since $\omega_{-j}(i)+\omega_{j}(i^*)=0$ for all $i$, we obtain that $d\in\cbra{0,1,2}$. Suppose $d=0$. Then $v_j(i)$ must attain its minimal possible value, namely $v_j=\omega_{j}$. However, this contradicts condition (P2). Similarly, $d=2$ is also impossible. We conclude that $d=1$.

    (2) By \eqref{ai} in the proof of Lemma \ref{permp1p2}, the $k[[t]]$-lattice $w\lambda_j$ is \[k[[t]]\pair{t^{r_1}e_1,\ldots, t^{r_{2n}}e_{2n} }, \]
        where $(r_1,\ldots,r_{2n})=w\omega_j\in\BZ^{2n}$. Then we obtain that $$\CF^w_j=w\lambda_j/t\lambda_j\simeq  k\pair{\varepsilon_i\ |\ \mu_j^w(i) =0},$$
        where the last isomorphism is induced by the isomorphism $\Lambda_j\simeq\CO^{2n}$ in \eqref{fi}.
        We note that as $w$ is naively-permissible, $\CF^w_j$ is indeed an $n$-dimensional subspace of $k^{2n}$.
\end{proof}

\begin{defn}
	(cf. \cite[\S 7.5]{smithling2011topological}) Let $w\in \wt{W}$ be naively-permissible. For $j\in 2n\BZ\pm I$, we say that the vector $\mu_j^w\in \BZ^{2n}$ in Lemma \ref{lem-perm} (2) is \dfn{totally isotropic} if $$\mu_j^w+(\mu_j^w)^*=\mathbf 1.$$  
\end{defn}
For a naively-permissible $w$, set \begin{flalign*}
		 E_j^w\coloneqq \cbra{i\ |\ \mu_j^w(i)=0},\ j\in 2n\BZ\pm I.
	\end{flalign*}
	Then $E_j^w$ is a subset of $\sbra{1,2n}$ of cardinality $n$. By definition, $\mu_j^w$ is totally isotropic if and only if $E_j^w=(E_j^w)^\perp$ (notation as in \S \ref{subsec-notation-intro}).

\begin{prop}\label{prop-pmperm}
	Let $w\in\wt{W}$ be naively-permissible. Then $w$ defines a point $(\CF_j^w)_{2n\BZ\pm I}\in \RM^\naive_I(k)$ by Lemma \ref{lem-perm}. Then the following are equivalent. \begin{enumerate}[label=(\roman*)]
		\item $w$ is $\pm$-permissible (see Definition \ref{defn-permi}).
		\item Denote $N_+\coloneqq k\pair{\varepsilon_1,\ldots,\varepsilon_n}$ and $N_-\coloneqq k\pair{\varepsilon_1,\ldots,\varepsilon_{n-1},\varepsilon_{n+1}}$. If $\CF^w_j\sset k^{2n}$ is totally isotropic in $k^{2n}$, then $\CF^w_j$ and $N_\pm$ specify points on the same connected component of the orthogonal Grassmannian $\OGr(n,k^{2n})$.
		\item For all totally isotropic vectors $\mu_j^w$, we have $\mu_j^w\in W^\circ\mu_\pm=S_{2n}^\circ\mu_\pm$.
	\end{enumerate}   
\end{prop}
\begin{proof}
	See \cite[Proposition 7.4.3 and \S 7.5]{smithling2011topological}. Although the results in \loccit\ are stated for $I=\sbra{0,n}$, the same arguments carry over when $I$ is any subset.
\end{proof}

\begin{corollary} \label{coro-inWcirc}
	Let $w\in \wt{W}$.
	Suppose that $C_w$ is a Schubert cell in $\RM^\pm_{I,k}$. Then there exists some $w'\in \wt{W}^\circ$ such that $C_w=C_{w'}$.
\end{corollary}
\begin{proof}
	By Proposition \ref{lem-cell}, the assertion is clear when $I$ is of type \RNum{2} or \RNum{3}. Now assume that $I$ is of type \RNum{1}, i.e., $I$ contains $\cbra{0,n}$. 
	
	Note that $w$ is $\pm$-permissible by assumption on $w$. By definition of $\RM^\naive_{I}$, the subspaces $\CF_0^w$ and $\CF_{n}^w$ are totally isotropic in $k^{2n}$. By Proposition \ref{prop-pmperm}, we have $\mu_0^w, \mu_n^w\in S_{2n}^\circ\mu_\pm$. We claim that $w\in \wt{W}^\circ$. Write $w=t^ww_0\in\wt{W}= X_*(T)\rtimes S_{2n}^*$. It suffices to show that $w_0\in S_{2n}^\circ$, i.e., $w_0$ is an even permutation. We treat the case $\mu_+$; the case $\mu_-$ is completely analogous. 
	
	Denote by $$\varepsilon\colon \BZ^{2n}\ra \BZ/2\BZ$$ the map sending $v\in \BZ^{2n}$ to $\sum_{i=1}^nv(i)\mod 2$. Since $\mu_+= (1^{(n)},0^{(n)})$, the vector $\mu_0^w$ (resp. $\mu_n^w$) lies in $S_{2n}^\circ\mu_+$ if and only if $\varepsilon(\mu_0^w)=\varepsilon(\mu_+)=n\mod 2$ (resp. $\varepsilon(\mu_n^w)=n\mod 2$).   By definition, we have \begin{flalign*}
		  \mu_0^w=t^w \text{\ and\ } \mu_n^w=t^w+w_0\omega_n-\omega_n.
	\end{flalign*}
	It follows from the property $\mu_0^w,\mu_n^w\in S_{2n}^\circ\mu_+$ that $$\varepsilon(w_0\omega_n-\omega_n)=0.$$
	Since $\omega_n=((-1)^{(n)},0^{(n)})$, we obtain that $w_0$ is an even permutation (cf. the proof of Lemma \ref{lem-Wmperm}).
\end{proof}

\section{Topological flatness of spin local models} \label{sec-topoflat}
In this section, we prove Theorem \ref{intro-thmtopo}, namely that the spin local model is topologically flat over $\CO$ for any parahoric level structure. We first treat the pseudo-maximal parahoric case in \S \ref{subsec-pseud}, establishing Theorem \ref{intro-thmpseumax}. In \S \ref{subsec-genepara}, we then deduce topological flatness for an arbitrary parahoric level by reducing to the pseudo-maximal case, using the vertexwise criterion \cite{haines2017vertexwise} for admissible subsets.

\subsection{Pseudo-maximal parahoric case} \label{subsec-pseud}
Suppose\footnote{In fact, by Remark \ref{rmk-toI}, it is enough to consider the range $0\leq i\leq \lfloor n/2\rfloor$. } $I=\cbra{i}$, where $0\leq i\leq n$.

\begin{prop}\label{prop-i0n}
	Suppose $i=0$ or $n$. Then $\RM^\naive_i$ is isomorphic to $\OGr(n,2n)$ over $\CO$. In particular, $\RM^\naive_i$ is $\CO$-smooth, and hence is flat over $\CO$. 
\end{prop}
\begin{proof}
	From the discussion after Definition \ref{defn-permi}, for any $\CO$-algebra $R$, $\RM^\naive_i(R)$ is the set of locally direct summands $\CF\sset R^{2n}$ of rank $n$ satisfying $\CF=\CF^\perp$. Hence, we obtain $\RM^\naive_i\simeq \OGr(n,2n)$ over $\CO$.
\end{proof}

From now on, we assume that $i\neq 0,n$. 

For $I=\cbra{i}$, giving an $I$-face (or simply face) is equivalent to giving a pair $(v_i,v_{-i})$ of vectors in $\BZ^{2n}$ such that \begin{enumerate}
        \item $v_{-i}\geq v_i\geq v_{-i}-\mathbf 1 $; 
        \item $\Sigma v_i=\Sigma v_{-i}-2i$;
        \item there exists $d\in\BZ$ such that $v_i+v_{-i}^*=\mathbf d$.  
    \end{enumerate}  

\begin{lemma} \label{lem-Wmfaces}
   \begin{enumerate}
    \item The group $W'$ (see Definition \ref{defnWprime}) acts transitively on $\cbra{\text{$I$-faces}}$.
   	\item The stabilizer subgroup in $W'$ of $(\omega_i,\omega_{-i})$ is $W_i$ (see Lemma \ref{lem-cell}).
   	\item The map $w\mapsto w(\omega_i,\omega_{-i})$ induces a bijection \begin{flalign*} 
		 W'/W_i \simto \cbra{\text{$I$-faces}}.
	\end{flalign*} 
   \end{enumerate}
\end{lemma}
\begin{proof}
    (1) 	Let $(v_i,v_{-i})$ be a face. By Definition \ref{defn-face} (3), we have \begin{flalign*}
	v_{-i}-v_i = (\mathbf d-v_i^*)-v_i = \mathbf d-v_i-v_i^*.
	\end{flalign*}
	Thus, $v_{-i}-v_i=(v_{-i}-v_i)^*$. In particular, $v_{-i}-v_i$ is of the form $$(r_1,\ldots,r_{n-1},r_n,r_n,r_{n-1},\ldots,r_1)\in \BZ^{2n}.$$ By Definition \ref{defn-face} (1) and (2),  for $1\leq j\leq n$, we have $r_j\in\tcbra{0,1}$, and exactly $i$ of the $r_j$'s equal $1$. Therefore, there exists a $w_0\in S_{2n}^*$ such that \begin{flalign*}
		v_{-i}-v_i=w_0(1^{(i)},0^{(2n-2i)},1^{(i)})=w_0(\omega_{-i}-\omega_i).
	\end{flalign*}
	If $w_0$ is not in $S_{2n}^\circ$, then we replace $w_0$ by $w_0(1,2n)$. Then $w_0\in S^\circ_{2n}$ and $v_{-i}-v_i=w_0(\omega_{-i}-\omega_i)$. Set \begin{flalign*}
		  t^w\coloneqq v_i-w_0\omega_i \text{\ and\ } w\coloneqq t^ww_0.
	\end{flalign*}
	For $1\leq j\leq 2n$, we have \begin{flalign*}
		  t^w(j)+t^w(j^*) &=v_i(j)-(w_0\omega_i)(j) + v_i(j^*)-(w_0\omega_i)(j^*)\\ &= v_i(j)+d-v_{-i}(j) -(w_0\omega_i)(j)-(w_0\omega_i)(j^*)\\ &=d- (v_{-i}-v_{i})(j)+w_0(\omega_{-i}-\omega_i)(j) \\ &=d.
	\end{flalign*}
	If $\varepsilon(t^w)=0$, then $w\in W'$, we are done. Suppose $\varepsilon(t^w)=1$. Since $i\neq 0,n$, the element (cf. Proposition \ref{prop-paraIndex}) \begin{flalign*}
		   \tau_1\coloneqq (-1,0^{(2n-2)},1)\cdot (1,2n)(n,n+1)\in \wt{W}^\circ
	\end{flalign*} 
	satisfies $\tau_1(\omega_i,\omega_{-i})=(\omega_i,\omega_{-i})$ and $\varepsilon(\tau_1)=1$. Hence, $\varepsilon(w\tau_1)=0$, and so $w\tau_1\in W'$. Furthermore, we have $w\tau_1(\omega_i,\omega_{-i})=(\omega_i,\omega_{-i})$. This shows that $W'$ acts transitively on the set of $\cbra{i}$-faces. 

    (2)
    Denote $H=\tcbra{w\in W' \ |\ w(\omega_i,\omega_{-i})=(\omega_i,\omega_{-i}) }.$ 
	Note that $a_i$ is the midpoint of $\omega_i$ and $\omega_{-i}$. Since $W'$ acts via affine transformations, we obtain that $H$ stabilizes $a_i$. Let $w=t^ww_0\in H$. Then \begin{flalign}
		  wa_i=t^w+w_0a_i =a_i.  \label{=ai}
	\end{flalign}
	In particular, $t^w=a_i-w_0a_i\in \BZ^{2n}$. As in the proof of Lemma \ref{lem-Wmperm}, it follows that the permutation $w_0$ permutes both $A_i$ and $B_i$. Then it is straightforward to check that $t^w+(t^w)^*=\mathbf 0$. Since $w\in H$, we have $w\in W_\aff$, and hence \begin{flalign*}
		H=\cbra{w\in W_\aff\ |\ w(\omega_i,\omega_{-i})=(\omega_i,\omega_{-i})}.
	\end{flalign*}
	As $H$ fixes $a_i$, we have 
	$H\sset W_i$. Conversely, suppose $w=t^ww_0\in W_i$, i.e., $wa_i=a_i$ and $w\in W_\aff$. By \eqref{=ai}, we have $t^w=a_i-w_0a_i$. It follows that \begin{flalign*}
		 w\omega_i &= t^w+w_0\omega_i\\ &= a_i-w_0a_i +w_0\omega_i\\ &= ((-1/2)^{(i)},0^{(2n-2i)},(1/2)^{(i)}) +w_0((-1/2)^{(i)},0^{(2n-2i)},(-1/2)^{(i)}).
	\end{flalign*}
	Since $w_0\in S_{2i}^\circ\times S_{2n-2i}^\circ$ by Lemma \ref{lem-Wmperm}, we have \begin{flalign*}
		w\omega_i = ((-1/2)^{(i)},0^{(2n-2i)},(1/2)^{(i)}) +((-1/2)^{(i)},0^{(2n-2i)},(-1/2)^{(i)}) =\omega_i.
	\end{flalign*}
	Similarly, we obtain that $w\omega_{-i}=\omega_{-i}$. Hence, $H=W_i$. 
	
	(3) It immediately follows from (1) and (2).   
\end{proof}

\begin{remark} \label{rmk-Wm}
Lemma \ref{lem-Wmfaces} is essentially contained in \cite[Lemma 8.13.6]{smithling2014topological}. We give a direct proof here for the reader's convenience.
\end{remark}

For $I=\cbra{i}$, a $k$-point in $\RM^\naive_I(k)$ amounts to a pair $(\CF_i,\CF_{-i})$ of subspaces in $k^{2n}$ satisfying conditions LM$1'$ to LM$3'$ (see the paragraph after Definition \ref{defn-permi}). 
\begin{defn}
    Let $E$ be a subset of $\sbra{1,2n}$ with cardinality $n$. \begin{enumerate}
        \item We say that a subspace $\CF\sset k^{2n}$ is \dfn{given by} $E$ if $\CF$ has a $k$-basis consisting of $\varepsilon_j$ for $j\in E$.
        \item We say that $E$ is \dfn{naively-permissible} if, for every pair $\cbra{j,j^*}\sset A_i$ (notation as in Lemma \ref{lem-Wmperm}), the pair is not contained in $E$, and for every pair $\cbra{j,j^*}\sset B_i$, at least one element lies in $E$.
    \end{enumerate} 
\end{defn}

Let $w\in \wt{W}^\circ$ be naively-permissible\footnote{If $I=\cbra{i}$, then $I$ is necessarily of type \RNum{2} or \RNum{3}. We may assume that $w\in \wt{W}^\circ$ by Proposition \ref{lem-cell}. }. 
By Lemma \ref{lem-perm}, the subspace $\CF^w_i$ (resp. $\CF^w_{-i}$) is given by $$E^w\coloneqq E^w_i=\cbra{j\ |\ \mu_i^w(j)=0\text{\ and\ } 1\leq j\leq 2n} $$ (resp. $(E^w)^\perp$). Here, $(E^w)^\perp$ denotes the complement of the subset $\cbra{j^*\ |\ j\in E^w}$.
 
 Clearly, for $u\in W_i$, we have $$E^w=E^{wu}.$$ 

\begin{lem}\label{lem-Eperm}
    Let $w\in\wt{W}^\circ$ be naively-permissible. Then $E^w$ is naively-permissible.
\end{lem}
\begin{proof}
    Recall $v_{\pm i}\coloneqq w\omega_{\pm i}$. Note that $(v_i,v_{-i})$ is a face. By Definition \ref{defn-face} and Lemma \ref{lem-perm} (1), we have \begin{flalign*}
        1-v_i(j^*)\geq v_i(j)\geq -v_i(j^*)\ (1\leq j\leq 2n) \text{\ and\ }\Sigma v_i=n-i.
    \end{flalign*}
    It follows that for $j\in A_i$, we have \begin{equation}
       \begin{split}
           \mu_i^w(j)+\mu_i^w(j^*) &= v_i(j)+v_i(j^*)-\omega_i(j)-\omega_i(j^*)\\ &=v_i(j)+v_i(j^*)+1\in \cbra{1,2}. \label{mu3}
       \end{split}
    \end{equation}
    Similarly, for $j\in B_i$, we have \begin{flalign}
        \mu_i^w(j)+\mu_i^w(j^*)\in \cbra{0,1}.  \label{mu4}
    \end{flalign}  As $w$ is naively-permissible, we have $\mu_i^w(j)\in\cbra{0,1}$ for all $j$ by condition (P1). Hence, the formula in \eqref{mu3} implies that no pairs $\cbra{j,j^*}\sset A_i$ are contained in $E^w$, and the formula \eqref{mu4} implies that at least one element in $\cbra{j,j^*}\sset B_i$ lies in $E^w$. It proves that $E^w$ is naively-permissible. 
\end{proof}

\begin{lem}\label{lem-orbits4}
    There are precisely $\min{\cbra{i,n-i}}+4$ orbits of naively-permissible subsets under the action of $W_i=S_{2i}^\circ\times S_{2n-2i}^\circ$.
\end{lem}
\begin{proof}
    Let $E$ be a naively-permissible subset. Let $a_1$ (resp. $a_0$) denote the number of pairs in $A_i$ contributing one (resp. zero) element in $E$. Let $b_2$ (resp. $b_1$) denote the number of pairs in $B_i$ contributing two (resp. one) elements in $E$. Since $E$ is naively-permissible, we have \begin{align*}
        \max{\cbra{0,2i-n}}\leq a_1\leq i,\ a_0=i-a_1,\ b_1=n-2i+a_1,\ b_2=i-a_1.
    \end{align*}
    Note that the lower bound for $a_1$ is required to guarantee $b_1\geq 0$.
    
    Let $\ell$ be an integer satisfying $\max{\cbra{0,2i-n}}\leq \ell\leq i$. We say $E$ is of type $\ell$ if $a_1(E)=\ell$. It is clear that for $h\in W_i$, the (naively-permissible) subsets $h(E)$ and $E$ have the same type. In the proof of Lemma \ref{lem-Wmperm}, we see that $h\in W_i$ decomposes into a (reduced) product of $\tau_{j'j}=(j'j)(j'^*j^*)$ and $\tau_j=(jj^*)$, with an even number of $\tau_j$. If $\cbra{j',j'^*}\cap E\neq \emptyset$, then $\cbra{j,j^*}\cap (\tau_{j'j}E)\neq \emptyset$. If $j\in E$, then $j^*\in (\tau_jE)$. 
    
    Suppose $E$ is of type $\ell=i$. Then $E$ picks exactly one element from each pair $\cbra{j,j^*}$ in $A_i\cup B_i$. Denote \begin{flalign*}
        r_1(E)\coloneqq \# (E\cap\cbra{1,\ldots,i}) \text{\ and\ } r_2(E)\coloneqq \# (E\cap\cbra{i+1,\ldots,n}).
    \end{flalign*}
    Let $\CS$ be the set of naively-permissible subsets. By the action of $\tau_{j'j}$ and $\tau_j$ on $E$, we obtain that  \begin{align}
        \cbra{E\in\CS\ |\ a_1(E)=i, (-1)^{r_1(E)}=\pm 1, (-1)^{r_2(E)}=\pm 1 }  \label{ESm}
    \end{align}
    gives four $W_i$-orbits in $\CS$. 

    Suppose $\max{\cbra{0,2i-n}}\leq \ell<i$. Let $E,E'$ be two naively-permissible subsets of type $\ell$. Using the action of $\tau_{j'j}$, we may assume that the pairs $\cbra{j,j^*}$ contributing one (resp. two) element(s) in $E$ and $E'$ are the same. Since $\ell<i$, there exists a pair $\cbra{j_0,j_0^*}\sset A_i$ and a pair $\cbra{j'_0,j_0'^*}\sset B_i$ such that $\cbra{j_0,j_0^*}\cap E=\emptyset$ and $\cbra{j_0',j_0'^*}\sset E$. Then the transpositions $(j_0j_0^*)$ and $(j_0'j_0'^*)$ act trivially on $E$. Using the action of $(j_0j_0^*)(jj^*)$ ($j\in A_i$) and $(j_0'j_0'^*)(jj^*)$ ($j\in B_i$), we may assume that $E=E'$. Hence, the set \begin{flalign}
    	\cbra{E\in\CS\ |\ a_1(E)=\ell}  \label{ESl}
    \end{flalign} is a $W_i$-orbit.
    
    From the above discussion, we conclude that there are precisely $\min{\cbra{i,n-i}}+4$ $W_i$-orbits in $\CS$.  
\end{proof}

For $\ell=i$, define \begin{flalign*}
	E^i_1&\coloneqq \sbra{1,n},\ E_2^i \coloneqq [1,i-1]\cup [i+1,n-1]\cup [n+1,2n+1-i], 
	\\ E^i_3&\coloneqq \sbra{1,n-1}\cup \cbra{n+1},\
	E_4^i\coloneqq \sbra{1,i-1}\cup [i+1,n]\cup\cbra{2n+1-i}.
\end{flalign*}
By the proof of Lemma \ref{lem-orbits4}, $E^i_d$ for $1\leq d\leq 4$ gives a representative for each $W_i$-orbit in \eqref{ESm}. Similarly, for $\max{\cbra{0,2i-n}}\leq \ell<i$, the set \begin{flalign*}
	E_1^\ell\coloneqq \sbra{i+1-\ell,n+i-\ell}
\end{flalign*}
is a representative for the $W_i$-orbit in \eqref{ESl}. 

For a subset $E\sset \sbra{1,2n}$, define $$kE\coloneqq k\pair{\varepsilon_j\ |\ j\in E}\sset k^{2n}. $$
The above naively-permissible subsets $E^\ell_d$ determine $\min{\cbra{i,n-i}}+4$ points $(\CF^\ell_{i,d},\CF^\ell_{-i,d})$ in $\RM^\naive_i(k)$, via the following subspaces of $k^{2n}$:
\begin{equation}
	\begin{split}
		\CF_{i,d}^\ell\coloneqq kE_d^\ell \text{\ and\ } \CF_{-i,d}^\ell\coloneqq k(E_d^\ell)^\perp.
	\end{split}  \label{CF1234}
\end{equation} 
% \begin{equation}
%	\begin{split}
%		\CF^m_1 &\coloneqq k\pair{f_1,\ldots,f_{2m}},\  \CF^m_2\coloneqq k\pair{f_1,\ldots,f_{m-1},f_{m+1},\ldots,f_{2m-1},f_{2m+1},f_{3m+1}}, 
%		\\ \CF^m_3&\coloneqq k\pair{f_1,\ldots,f_m,f_{m+1},\ldots,f_{2m-1},f_{2m+1}}, 
%		\\ \CF^m_4&\coloneqq k\pair{f_1,\ldots,f_{m-1},f_{m+1},\ldots,f_{2m},f_{3m+1}}, 
%		\\ \CF^\ell_1 &\coloneqq k\pair{f_{m+1-\ell}, f_{m+2-\ell},\ldots, f_{3m-\ell} },\ 0\leq \ell< m.
%	\end{split}   \label{CF1234}
%\end{equation} 
Here, $\max{\cbra{0,2i-n}}\leq \ell\leq i$, with $1\leq d\leq 4$ when $\ell=i$, and $d=1$ when $\ell<i$.

Using Proposition \ref{prop-pmperm} (ii) and \cite[\S 7.1.4]{pappas2009local}, the following lemma is straightforward. 

\begin{lem}\label{lem-cells1}
    If $\ell=i$, then we have $(\CF^i_{i,d},\CF^i_{-i,d})\in\RM^+_i(i)$ for $d=1,2$, and $(\CF^i_{i,d},\CF^i_{-i,d})\in \RM^-_i(i)$ for $d=3,4$. If $\max\cbra{0,2i-n}\leq \ell<i$, then $(\CF^\ell_{i,1},\CF^\ell_{-i,1})\in \RM^+_i(\ell)\cap \RM^-_i(\ell)$.
\end{lem}

\begin{corollary}\label{coro-bijection}
    The Schubert cells in $\RM^\naive_{i,k}$ are in bijection with the orbits of naively-permissible subsets of $\sbra{1,2n}$ under the action of $W_i\simeq S_{2i}^\circ\times S_{2n-2i}^\circ$. Each Schubert cell in $\RM^\naive_{i,k}$ admits a representative of the form $(\CF^\ell_{i,d},\CF^\ell_{-i,d})$ in \eqref{CF1234}.  In particular, there are precisely $\min{\cbra{i,n-i}}+4$ Schubert cells in $\RM^\naive_{i,k}$. 
\end{corollary}
\begin{proof}
    Let $u=t^uu_0\in W_i$. Then $u\omega_i=\omega_i$ by Lemma \ref{lem-Wmfaces}. If $w$ is naively-permissible and $j\in E^w$, then \begin{flalign*}
    	uw\omega_i(u_0(j)) &=w\omega_i(j)+t^{u}(u_0(j))\\ &=\omega_i(j)+t^u(u_0(j)) \text{\quad (since $j\in E^w$)}\\ &= \omega_i(u_0(j)) \text{\quad (since $u\omega_i=\omega_i$).}
    \end{flalign*} 
    Hence, $u_0(j)\in E^{uw}$. Then $w\mapsto E^w$ defines a $W_i$-equivariant map \begin{flalign*}
    	\tcbra{\text{Naively-permissible elements in $\wt{W}^\circ$}}/W_i \ra \CS=\cbra{\Centerstack[l]{Naively-permissible subsets\\ in $\cbra{1,\ldots,2n}$}}.
    \end{flalign*}
    By Lemma \ref{lem-Wmfaces}, this map restricts to an injection \begin{flalign*}
    	\tcbra{\text{Naively-permissible elements in $W'$}}/W_i \hookrightarrow \CS, 
    \end{flalign*}
    and hence an injection on (left) $W_i$-orbit sets \begin{flalign*}
    	W_i\backslash\tcbra{\text{Naively-permissible elements in $W'$}}/W_i \hookrightarrow W_i\backslash \CS.
    \end{flalign*}
    By Lemma \ref{lem-cell}, the Schubert cells in $\RM^\naive_{i,k}$ are in bijection with the left-hand orbit set. Thus, the number of Schubert cells in $\RM^\naive_k$ is at most $|W_i\backslash \CS|=\min{\cbra{i,n-i}}+4$. On the other hand, by \eqref{CF1234}, there are at least $\min{\cbra{i,n-i}}+4$ Schubert cells in $\RM^\naive_{i,k}$. Then we complete the proof. 
\end{proof}
\begin{remark}
	The proof of Corollary \ref{coro-bijection} also implies that for any naively-permissible subset $E$, there exists a naively-permissible $w\in \wt{W}^\circ$ such that $E=E^w$.
\end{remark}

\begin{prop}\label{prop-liftm4}
    Each of the $\min{\cbra{i,n-i}}+4$ points $(\CF^\ell_{i,d},\CF^\ell_{-i,d})\in \RM^\naive_i(k)$ has a lift to the generic fiber $\RM^\naive_{i,F}$.
\end{prop}
\begin{proof}
%    By Proposition \ref{prop-Jform} and Remark \ref{isomconjugate}, we may assume that $i\leq \lfloor n/2\rfloor$. Then $i\leq n-i$.
    
    Note that for any $\CO$-algebra $R$, $\RM^\naive_i(R)$ is the set of pairs $(\CF_i,\CF_{-i})$ of $R$-modules such that 
    \begin{enumerate}[label=(\roman*)]
    	\item $\CF_i,\CF_{-i}$ are locally direct summands of rank $n$ in $R^{2n}$;
    	\item $\CF_{-i}=\CF^\perp_i$;
    	\item Denote \begin{flalign*}
    		\lambda_1\coloneqq \begin{pmatrix}
    			\pi I_i & & \\ &I_{2n-2i} & \\ & &\pi I_i
    		\end{pmatrix} \text{\ and\ } \lambda_2\coloneqq \begin{pmatrix}
    			I_i & & \\ &\pi I_{2n-2i} & \\ & &I_i
    		\end{pmatrix},
    	\end{flalign*}
    	which are viewed as $R$-linear maps $R^{2n}\ra R^{2n}$. We require $\lambda_1(\CF_{-i})\sset \CF_i$ and $\lambda_2(\CF_i)\sset \CF_{-i}$. 
    \end{enumerate} 
    Suppose that $\ell=i$. Define \begin{flalign*}
        \wt{\CF}^i_{i,1}\coloneqq \CO\pair{\varepsilon_1,\ldots,\varepsilon_{n} } \text{\ and\ } \wt{\CF}^i_{-i,1}\coloneqq \CO\pair{\varepsilon_1,\ldots,\varepsilon_{n} }=\wt{\CF}^i_{i,1}
    \end{flalign*}
    Then $(\wt{\CF}^i_{i,1},\wt{\CF}^i_{-i,1})\in \RM^\naive_i(\CO)$ lifts the point $(\CF^i_{i,1},\CF^i_{-i,1})$.
    Similarly, we can lift the points $(\CF^i_{i,d},\CG^i_{-i,d})$ for $d=2,3,4$.

    Suppose that $\max{\cbra{0,2i-n}}\leq \ell<i$. Let $L=F(\sqrt{\pi})$ be a quadratic extension of $F$. Define \begin{flalign*}
        \wt{\CF}^\ell_{i,1}\coloneqq \CO_L &\langle \varepsilon_{i+1-\ell},\ldots,\varepsilon_{n-i+\ell}, \varepsilon_{n-i+\ell+1}+\sqrt{\pi}\varepsilon_{i-\ell}, \varepsilon_{n-i+\ell+2}+\sqrt{\pi}\varepsilon_{i-\ell-1},\ldots, \\ &\varepsilon_{n}+\sqrt{\pi}\varepsilon_1, \varepsilon_{n+1}-\sqrt{\pi}\varepsilon_{2n}, \varepsilon_{n+2}-\sqrt{\pi}\varepsilon_{2n-1},\ldots,\varepsilon_{n+i-\ell}-\sqrt{\pi}\varepsilon_{2n-i+\ell+1} \rangle.
    \end{flalign*}
    We have \begin{flalign*}
    	\wt{\CF}_{-i,1}^\ell = &\CO_L \langle \varepsilon_1+\sqrt{\pi}\varepsilon_{n},\varepsilon_2+\sqrt{\pi}\varepsilon_{n-1},\ldots,\varepsilon_{i-\ell}+\sqrt{\pi}\varepsilon_{n-i+\ell+1},\varepsilon_{i-\ell+1}, \varepsilon_{i-\ell+2}, \\ &\ldots,\varepsilon_{n-i+\ell}, \varepsilon_{2n-i+\ell+1}-\sqrt{\pi}\varepsilon_{n+i-\ell}, \varepsilon_{2n-i+\ell+2}-\sqrt{\pi}\varepsilon_{n+i-\ell-1}, \ldots, \varepsilon_{2n}-\sqrt{\pi}\varepsilon_{n+1}\rangle.
    \end{flalign*}
    Then it is easy to check that $(\wt{\CF}^\ell_{i,1},\wt{\CF}^\ell_{-i,1})\in \RM^\naive_i(\CO_L)$ and lifts $(\CF^\ell_{i,1},\CF^\ell_{-i,1})$.
\end{proof}

\begin{corollary}\label{coro-topoflat}
    The naive local model $\RM^\naive_i$, and hence $\RM^\pm_i$, is topologically flat. Moreover, the underlying topological space of $\RM^\naive_i$ is the union of those of $\RM^+_i$ and $\RM^-_i$. 
\end{corollary}
\begin{proof}
    Since the points $(\CF^\ell_{i,d},\CF^\ell_{-i,d})$ are representatives for (all) Schubert cells in $\RM^\naive_{i,k}$ by Corollary \ref{coro-bijection}, it follows from Proposition \ref{prop-liftm4} that all points in $\RM^\naive_{i,k}$ can be lifted to the generic fiber. Hence, $\RM^\naive_i$ is topologically flat. 
    
    The assertion that $\RM^\naive_i$ is topologically the union $\RM^+_i\cup \RM^-_i$ follows easily from Lemma \ref{lem-cells1}.
\end{proof} 

\begin{defn}\label{defn-Mell}
    Let $\iota\colon \Lambda_i\ra \pi\inverse\Lambda_{-i}=\Lambda_{2n-i}$ denote the natural inclusion map (and its base change). 
     For an integer $\ell$, denote by $\RM^\pm_i(\ell)\sset \RM^\pm_{i,k}$ the locus where $\iota(\CF_i)$ has rank $\ell$.
\end{defn}

%\begin{lemma}
%	If $\RM^\pm_i(\ell)$ is non-empty, then $\max{\cbra{0,2i-n}}\leq \ell\leq i$.
%\end{lemma}
%\begin{proof}
%	Let $(\CF_i,\CF_{-i})$ be a $\ol{k}$-point of $\RM^\naive_i$. Here, we use the interpretation of $\RM^\naive_i(\ol{k})$ in the proof of Proposition \ref{prop-liftm4}. We have $\lambda_1(\CF_{-i})\sset \CF_i$ and $\lambda_2(\CF_i)\sset \CF_{-i}$, where \begin{flalign*}
%		  \lambda_1\coloneqq \begin{pmatrix}
%    			0 & & \\ &I_{2n-2i} & \\ & &0
%    		\end{pmatrix} \text{\ and\ } \lambda_2\coloneqq \begin{pmatrix}
%    			I_i & & \\ &0 & \\ & &I_i
%    		\end{pmatrix}.
%	\end{flalign*} Denote $W\coloneqq \lambda_2(\CF_i)\sset \CF_{-i}$ and $K\coloneqq \ker\lambda_2|_{\CF_i}\sset \CF_i$. By assumption, we have $$\dim W\coloneqq \dim_{\ol{k}}W=\ell \text{\ and\ } \dim K=n-\ell. $$
%	Write $V_1\coloneqq \ol{k}\pair{\varepsilon_1,\ldots,\varepsilon_i,\varepsilon_{2n+1-i},\ldots,\varepsilon_{2n}}$ and $V_2\coloneqq \ol{k}\pair{\varepsilon_{i+1}\ldots,\varepsilon_{2n-i}}$. Then $$\ol{k}^{2n}=V_1\perp V_2$$ is an orthogonal sum with respect to the split symmetric form on $\ol{k}^{2n}$ corresponding to the anti-diagonal matrix $H_{2n}$.
%	Note that $K$ lies in $V_2$.  Denote by $K'$ the orthogonal complement of $K$ in $V_2$. Then $$\dim K'=\dim V_1-\dim K=(2n-2i)- (n-\ell)=n-2i+\ell.$$
%	Since $\CF_i=K\oplus$
%\end{proof}

By standard properties of the rank function (cf. \cite[Proposition 16.18]{gortzwedhorn1}), each $\RM^\pm_i(\ell)$ is a locally closed subscheme of $\RM^\pm_{i,k}$, and the closure relation is governed by rank: $\RM^\pm_i(\ell')$ is contained in the Zariski closure $\ol{\RM^\pm_i(\ell)}$ if and only if $\ell'\leq \ell$.

Let $\sG_i$ denote the group scheme of similitude automorphisms of the self-dual lattice chain $\Lambda_{\cbra{i}}$. By Lemma \ref{lemstable}, the $\sG^\circ_{i,k}$-action preserves $\RM^\pm_{i,k}$. Since $\sG_i$ is, by definition, compatible with the map $\iota$, we have: for any $g\in \sG_{i,k}^\circ$, the submodules $\iota(\CF_i)$ and $\iota(g\CF_i)$ have the same rank. Hence, each stratum $\RM^\pm_i(\ell)$ is stable under $\sG^\circ_{i,k}$. Via the identification \eqref{sGisom}, the $L^+\CP_i^{\circ}$-action preserves $\RM^\pm_{i,k}\sset \sFl_i$. Thus, every $\RM^\pm_i(\ell)$ is a union of Schubert cells.

\begin{prop}\label{prop-stratificationMpm}
    There exists a stratification of the reduced special fiber \begin{flalign}
        (\RM^\pm_{i,k})_\red=\coprod_{\ell=\max\cbra{0,2i-n}}^i \RM^\pm_i(\ell), \label{stratification}
    \end{flalign}
    where the top stratum $\RM^\pm_i(i)$ decomposes into exactly two Schubert cells (of dimension $n(n-1)/2$), and each lower stratum $\RM^\pm_i(\ell)$ for $\max\cbra{0,2i-n}\leq \ell<i$ is a single Schubert cell. Consequently, $\RM^\pm_{i,k}$ has two irreducible components whose intersection is irreducible.
\end{prop}
\begin{proof}
    The rank stratification gives the decomposition \begin{flalign*}
        (\RM^\pm_{i,k})_\red=\coprod_{\ell\geq 0} \RM^\pm_i(\ell), 
    \end{flalign*}
    If the point $(\CF_{i,d}^\ell,\CF_{-i,d}^\ell)$ in Lemma \ref{lem-cells1} lies in $\RM^\pm_{i,k}$, then it lies in $\RM_i^\pm(\ell)$ by definition. Since $\RM^\pm(\ell)$ is a union of Schubert cells, the stated decomposition \eqref{stratification} follows from Corollary \ref{coro-bijection} and Lemma \ref{lem-cells1}. In particular, $\RM^\pm_i(\ell)$ is empty if $\ell<\max{\cbra{0,2i-n}}$ or $\ell>i$. 
    
    Write $\RM^\pm_i(i)=C_1\cup C_2$ as a union of Schubert cells. Since the top stratum $\RM^\pm_i(i)$ is open, each cell $C_j$ for $j=1,2$ is necessarily open. Hence, the closures $\ol{C}_1$  and $\ol{C}_2$ are irreducible components of $\RM^\pm_{i,k}$. The intersection $\ol{C}_1\cap\ol{C}_2$ is given by \begin{flalign*}
        \RM^\pm_i(<i)=\coprod_{\ell<i}\RM^\pm_i(\ell),
    \end{flalign*}
    which is irreducible, since the top stratum\footnote{Recall that we assume $i\neq 0,n$.} $\RM^\pm_i(i-1)$ is a single Schubert cell.
    
    By Corollary \ref{coro-topoflat}, we have $\dim \RM^\naive_i=\dim \RM^{\pm\loc}_i=n(n-1)/2$ (see Proposition \ref{propiso}). Hence, the dimension of the Schubert cell in $\RM^\pm_i(i)$ is also $n(n-1)/2$. 
\end{proof}

\begin{remark}
	If $i=0$ or $n$, then $\RM^\pm_{i,k}=(\RM^\pm_{i,k})_\red=\RM_{i,k}^\pm(i)$ is a single Schubert cell. 
\end{remark}

\subsection{General parahoric case} \label{subsec-genepara}
Now we let $I\sset [0,n]$  be any non-empty subset. By construction, for each $i\in I$, there is a natural $L^+\CP_I$-equivariant  projection map \begin{flalign}
	  q_i\colon LG/L^+\CP_I\ra LG/L^+\CP_i,.  \label{qiproj}
\end{flalign}
By the moduli description and the definition of the spin condition, we obtain the following.
\begin{lemma}\label{lem-intersection}
	The scheme $\RM^\pm_I$ equals the schematic intersection $\bigcap_{i\in I}q_i\inverse(\RM^\pm_i)$. Here, $q\inverse_i(\RM^\pm_i)$ denotes the pullback of $\RM^\pm_i$ along $q_i$. 
%	\begin{flalign*}
%		\bigcap_{i\in I}q_i\inverse(\RM^\pm_i).
%	\end{flalign*}
\end{lemma}

By Proposition \ref{prop-Jform}, the subset $I$ corresponds to a parahoric subgroup $P_I^\circ$ stabilizing a facet $\ff_I\sset \fa$. Let $J\sset \sI$ be the set of vertices in the closure $\ol{\ff}_I$. Concretely, \begin{flalign*}
	   J=J_0\cup J_1,
\end{flalign*}
where $J_0\coloneqq I\cap\cbra{0,2,3,\ldots,n-2,n}$, and \begin{flalign*}
	J_1\coloneqq \begin{cases}
		\cbra{0,0'} \quad &\text{if $1\in I$ and $n-1\notin I$};\\ \cbra{n,n'} &\text{if $n-1\in I$ and $0\notin I$};\\ \cbra{0,0',n,n'} &\text{if $1,n-1\in I$}; \\ \emptyset &\text{if $1,n-1\notin I$}.
	\end{cases}
\end{flalign*}

By Corollary \ref{coro-inWcirc}, any Schubert cell in $\RM^\pm_{I,k}$ is of the form $C_w$ for $w\in\wt{W}^\circ$. Denote by \begin{flalign}
	\Perm^\pm\coloneqq \Perm^\pm_I \sset \wt{W}^\circ \label{perm46}
\end{flalign}  the subset consisting of $w$ such that $C_w$ lies in $\RM^\pm_{I,k}$. Denote by \begin{flalign*}
        \Adm(\mu_\pm)\coloneqq \tcbra{w\in\wt{W}^\circ\ |\ w\leq \sigma t^{\mu_\pm}\sigma\inverse \text{\ for some $\sigma\in W^\circ$} }
        \end{flalign*}
    the $\mu_\pm$-admissible subset.

\begin{lemma} \label{lem-perm-admi}
	For any $i\in I$, we have \begin{flalign*}
		W_i\backslash \Perm_{i}^\pm/W_i = W_i\backslash \Adm(\mu_\pm)/W_i.
	\end{flalign*}
\end{lemma}
\begin{proof}
   By Proposition \ref{prop-i0n} and Corollary \ref{coro-topoflat}, we obtain that $\RM^\naive_{i,k}$ and $\RM^{\pm\loc}_{i,k}$ (see Definition \ref{Mlocdefin}) have the same topological space. By Proposition \ref{propiso} and a well-known fact (see e.g. \cite[Theorem 1.1]{pappas2013local}) about the schematic local model, we obtain the lemma. 
\end{proof}

For $i\in I\cup J$, there is an obvious projection map \begin{flalign*}
	\rho_i\colon W_I\backslash \wt{W}^\circ/W_I\ra W_i\backslash \wt{W}^\circ/W_i.
\end{flalign*}
Write $\Adm_i(\mu_\pm)\coloneqq W_i\backslash \Adm(\mu_\pm)/W_i$. We also use $\Perm_i^\pm$ to denote $W_i\backslash \Perm_i^\pm/W_i$ by abuse of notation. Since $q_i$ is $L^+\CP_I$-equivariant, we have \begin{flalign}
	W_I\backslash\Perm^\pm_I/W_I =\bigcap_{i\in I} \rho_i\inverse(\Perm_i^\pm)  \label{eq48}
\end{flalign}
by Lemma \ref{lem-intersection}.
Recall the following vertexwise criterion for the admissible sets.
\begin{thm}
	[{\cite[Theorem 1.5]{haines2017vertexwise}}] \label{thm-vertexcri}
	\begin{flalign*}
		W_I\backslash \Adm(\mu_\pm)/W_I = \bigcap_{j\in J}\rho_j\inverse (\Adm_j(\mu_\pm)).
	\end{flalign*}
\end{thm}

It follows that \begin{flalign*}
	 W_I\backslash\Perm^\pm_I/W_I &=\bigcap_{i\in I} \rho_i\inverse(\Perm_i^\pm) \text{\quad (by \eqref{eq48})}  \\ &=\bigcap_{i\in I}\rho_i\inverse(\Adm_i(\mu_\pm))  \text{\quad (by Lemma \ref{lem-perm-admi})} \\ &=\bigcap_{j\in J}\rho_j\inverse(\Adm_j(\mu_\pm)) \text{\quad (by Theorem \ref{thm-vertexcri})}.
\end{flalign*}  
By Theorem \ref{thm-vertexcri} again, the right-hand side is equal to $W_I\backslash\Adm(\mu_\pm)/W_I$. 

\begin{thm}\label{coro-mainresults}
	The spin local model $\RM^\pm_I$ is topologically flat over $\CO$ for any non-empty $I\sset [0,n]$. 
\end{thm}
\begin{proof}
	By the previous discussion, the Schubert cells in $\RM^\pm_{I,k}$ are indexed by the admissible set $W_I\backslash\Adm(\mu_\pm)/W_I$. In particular, $\RM^\pm_I$ and $\RM^{\pm\loc}_I$ have the same topological space, and hence $\RM^\pm_I$ is topologically flat. 
\end{proof}

By Remark \ref{rmk-toI}, this completes the proof of Theorem \ref{intro-thmtopo}.

%\bibliographystyle{plain} % We choose the "plain" reference style
%\bibliography{myref} % Entries are in the refs.bib file

\end{document}